\documentclass{article}
\RequirePackage[colorlinks,citecolor=blue,urlcolor=blue,linkcolor=blue]{hyperref}

\usepackage[final]{nips_2016}

\usepackage[ansinew]{inputenc}

\usepackage{doi}

\usepackage{amsmath}
\usepackage{mathtools}

\usepackage{amsfonts}
\usepackage{amsthm}
\usepackage{amssymb}

\usepackage{color}
\usepackage{url}
\usepackage{bm} % bold math fonts with \bm{math expr}
\usepackage{dsfont} % blackboard bold ones

\usepackage{graphicx}

\usepackage{adjustbox}
\usepackage{multirow}
\usepackage{longtable}

\usepackage{tikz}
\definecolor{mycolor1}{rgb}{0.105882,0.619608,0.466667}
\definecolor{mycolor2}{rgb}{0.85098,0.372549,0.00784314}
\definecolor{mycolor3}{rgb}{0.458824,0.439216,0.701961}
\definecolor{mycolor4}{rgb}{0.905882,0.160784,0.541176}
\definecolor{mycolor5}{rgb}{0.4,0.65098,0.117647}
\definecolor{mycolor6}{rgb}{0.65098,0.462745,0.113725}
\definecolor{mycolor7}{rgb}{0.901961,0.670588,0.00784314}
\definecolor{mycolor8}{rgb}{0.4,0.4,0.4}
\definecolor{mycolor9}{rgb}{0.301961,0,0.294118}
\definecolor{mycolor10}{rgb}{0.0313725,0.25098,0.505882}

\usetikzlibrary{calc}        

\makeatletter
\newif\ifmygrid@coordinates
\tikzset{/mygrid/step line/.style={line width=0.80pt,draw=gray!80},
         /mygrid/steplet line/.style={line width=0.25pt,draw=gray!80}}
\pgfkeys{/mygrid/.cd,
         step/.store in=\mygrid@step,
         steplet/.store in=\mygrid@steplet,
         coordinates/.is if=mygrid@coordinates}
\def\mygrid@def@coordinates(#1,#2)(#3,#4){%
    \def\mygrid@xlo{#1}%
    \def\mygrid@xhi{#3}%
    \def\mygrid@ylo{#2}%
    \def\mygrid@yhi{#4}%
}
\newcommand\DrawGrid[3][]{%
    \pgfkeys{/mygrid/.cd,coordinates=true,step=1,steplet=0.2,#1}%
    \draw[/mygrid/steplet line] #2 grid[step=\mygrid@steplet] #3;
    \draw[/mygrid/step line] #2 grid[step=\mygrid@step] #3;
    \mygrid@def@coordinates#2#3%
    \ifmygrid@coordinates%
        \draw[/mygrid/step line]
        \foreach \xpos in {\mygrid@xlo,...,\mygrid@xhi} {%
          (\xpos,\mygrid@ylo) -- ++(0,-3pt)
                              node[anchor=north] {$\xpos$}
        }
        \foreach \ypos in {\mygrid@ylo,...,\mygrid@yhi} {%
          (\mygrid@xlo,\ypos) -- ++(-3pt,0)
                              node[anchor=east] {$\ypos$}
        };
    \fi%
}
\makeatother

\usepackage{psfrag}

\usepackage{algorithm}
\usepackage{algpseudocode}
\algdef{SE}[DOWHILE]{Do}{doWhile}{\algorithmicdo}[1]{\algorithmicwhile\ #1}%

\usepackage{mathrsfs} % for \mathscr{ABC} style (curly calligraphy)

\newcommand{\remove}[1]{}
\newcommand{\removesafe}[1]{}

\usepackage{subfigure}

\newcommand{\transpose}{^\top\! }

\newcommand{\inner}[2]{\left\langle{#1},{#2}\right\rangle}
\newcommand{\innersmall}[2]{\langle{#1},{#2}\rangle}

\newcommand{\trace}{\mathrm{Tr}}
\newcommand{\Trace}{\mathrm{Tr}}

\newcommand{\im}{\mathrm{im}}

\newcommand{\Proj}{\mathrm{Proj}}

\newcommand{\T}{\mathrm{T}}

\newcommand{\Snn}{{\mathbb{S}^{n\times n}}}

\newcommand{\Rnn}{{\mathbb{R}^{n\times n}}}
\newcommand{\Rnp}{{\mathbb{R}^{n\times p}}}

\newcommand{\Rp}{{\mathbb{R}^{p}}}

\newcommand{\reals}{{\mathbb{R}}}

\newcommand{\Rn}{{\mathbb{R}^n}}
\newcommand{\Rm}{{\mathbb{R}^m}}

\newcommand{\grad}{\mathrm{grad}}
\newcommand{\Hess}{\mathrm{Hess}}

\newcommand{\diag}{\mathrm{diag}}
\newcommand{\D}{\mathrm{D}}

\newcommand{\calA}{\mathcal{A}}
\newcommand{\calN}{\mathcal{N}}
\newcommand{\calM}{\mathcal{M}}
\newcommand{\calC}{\mathcal{C}}

\newcommand{\calO}{\mathcal{O}}
\newcommand{\Id}{\operatorname{Id}} % need to be distinguishable from identity matrix
\newcommand{\rank}{\operatorname{rank}}
\newcommand{\nulll}{\operatorname{null}}

\newcommand{\lambdamin}{\lambda_\mathrm{min}}

\newtheorem{theorem}{Theorem} %[section]
\newtheorem{lemma}[theorem]{Lemma}
\newtheorem{proposition}[theorem]{Proposition}
\newtheorem{corollary}[theorem]{Corollary}

\newtheorem{definition}{Definition} %[section]

\title{The non-convex Burer--Monteiro approach works\\on smooth semidefinite programs}

\author{Nicolas Boumal$_\star$ \\ Department of Mathematics \\ Princeton University \\ \texttt{nboumal@math.princeton.edu} \And Vladislav Voroninski$_\star$ \\ Department of Mathematics \\ Massachusetts Institute of Technology \\ \texttt{vvlad@math.mit.edu} \And Afonso S.\ Bandeira \\ Department of Mathematics and Center for Data Science \\ Courant Institute of Mathematical Sciences, New York University \\ \texttt{bandeira@cims.nyu.edu}}

\newcommand{\1}{\mathbf{1}}

\begin{document}

\maketitle

{\let\thefootnote\relax\footnote{$\star$The first two authors contributed equally.}}
\setcounter{footnote}{0}

\begin{abstract}
%%%%	Computational tasks arising in machine learning are often naturally framed as optimization problems. When these are not tractable, it is common to resort to convex relaxations. Semidefinite programs (SDP's) are particularly popular in that context.
%%%	Semidefinite programs (SDP's) can be solved in polynomial time by interior point methods, but scalability can be an issue. To address this shortcoming, over a decade ago, Burer and Monteiro proposed to solve SDP's with few equality constraints via rank-restricted, non-convex surrogates. %, thus leveraging the fact that such SDP's admit low-rank solutions.
%%%	Remarkably, for some applications, local optimization methods seem to converge to global optima of these non-convex surrogates reliably.
%%%	Although some theory supports this empirical success, a complete explanation of it remains an open question.
%%%%	
%%%	In this paper, we consider a class of SDP's which includes applications such as max-cut, community detection in the stochastic block model, robust PCA, phase retrieval and synchronization of rotations. We show that the low-rank Burer--Monteiro formulation of SDP's in that class almost never has any spurious local optima.
%%%	
%%%%	identify a class of SDP's for which the non-convex low-rank problem has no spurious local optimizers, thus explaining why local optimization methods reach global optimizers.
	
Semidefinite programs (SDPs) can be solved in polynomial time by interior point methods, but scalability can be an issue. To address this shortcoming, over a decade ago, Burer and Monteiro proposed to solve SDPs with few equality constraints via rank-restricted, non-convex surrogates. Remarkably, for some applications, local optimization methods seem to converge to global optima of these non-convex surrogates reliably. Although some theory supports this empirical success, a complete explanation of it remains an open question. In this paper, we consider a class of SDPs which includes applications such as max-cut, community detection in the stochastic block model, robust PCA, phase retrieval and synchronization of rotations. We show that the low-rank Burer--Monteiro formulation of SDPs in that class almost never has any spurious local optima.
\vspace{2mm}
\\
\textcolor{blue}{This paper was corrected on April 9, 2018. Theorems~\ref{thm:masterthm} and~\ref{thm:approxtolerance} had the assumption that $\calM$~\eqref{eq:M} is a manifold. From this assumption it was stated that $\T_Y\calM = \{ \dot Y \in \Rnp : \calA(\dot Y Y\transpose + Y \dot Y\transpose) = 0 \}$, which is not true in general. To ensure this identity, the theorems now make the stronger assumption that gradients of the constraints $\calA(YY\transpose) = b$ are linearly independent for all $Y$ in $\calM$. All examples treated in the paper satisfy this assumption. Appendix~\ref{apdx:regularity} gives details.}
\end{abstract}

%\TODO{Have numerics. Worst case, they can go in the supplementary material at NIPS if we go in. Will make it easier in the conclusion and for the global argument if we are indeed faster. If not, cite Journeee instead of my own papers to support computations.}

\section{Introduction}

We consider semidefinite programs (SDPs) of the form
\begin{align}
f^* = \min_{X\in\Snn} \inner{C}{X} \quad \textrm{ subject to }  \quad \calA(X) = b, \ X \succeq 0,
\tag{SDP}
\label{eq:SDP}
\end{align}
where $\inner{C}{X} = \trace(C\transpose X)$, $C \in \Snn$ is the symmetric cost matrix, $\calA \colon \Snn \to \Rm$ is a linear operator capturing $m$ equality constraints with right hand side $b\in\Rm$ and the variable $X$ is symmetric, positive semidefinite.
Interior point methods solve~\eqref{eq:SDP} in polynomial time~\citep{nesterov2004introductory}. In practice however, for $n$ beyond a few thousands, such algorithms run out of memory (and time), prompting research for alternative solvers.

If~\eqref{eq:SDP} has a compact search space, then it admits a global optimum of rank at most~$r$, where $\frac{r(r+1)}{2} \leq m$~\citep{pataki1998rank,barvinok1995problems}.
Thus, if one restricts the search space of~\eqref{eq:SDP} to matrices of rank at most $p$ with $\frac{p(p+1)}{2} \geq m$, then the  globally  optimal value remains unchanged. This restriction is easily enforced by factorizing $X = YY\transpose$ where $Y$ has size $n\times p$, yielding an equivalent quadratically constrained quadratic program:
\begin{align}
q^* = \min_{Y\in\Rnp} \inner{CY}{Y} \quad \textrm{ subject to } \quad \calA(YY\transpose) = b.
\tag{P}
\label{eq:QCQP}
\end{align}
%Since~\eqref{eq:SDP} is a relaxation of~\eqref{eq:QCQP} (up to parameterization), $q^* \geq f^*$.
In general, \eqref{eq:QCQP} is non-convex, making it a priori unclear how to solve it globally. Still, the benefits are that it is lower dimensional than~\eqref{eq:SDP} and has no conic constraint.
%In the early 2000's,
This has motivated~\citet{sdplr,burer2005local} to try and solve~\eqref{eq:QCQP} using local optimization methods, with surprisingly good results. They developed theory in support of this observation (details below). About their results, %they write the following about their results:
\citet[\S3]{burer2005local} write (mutatis mutandis):
\begin{quote}
	``\emph{How large must we take $p$ so that the local minima of~\eqref{eq:QCQP} are guaranteed to map to global minima of~\eqref{eq:SDP}? Our theorem asserts that we need only\footnote{The condition on $p$ and $m$ is slightly, but inconsequentially, different in~\citep{burer2005local}.} $\frac{p(p+1)}{2} > m$ (with the important caveat that positive-dimensional faces of~\eqref{eq:SDP} which are `flat' with respect to the objective function can harbor non-global local minima).}''
\end{quote}
The caveat---the existence or non-existence of non-global local optima, or their potentially adverse effect for local optimization algorithms---was not further discussed.

In this paper, assuming $\frac{p(p+1)}{2} > m$, we show that if the search space of~\eqref{eq:SDP} is \emph{compact} and if the search space of~\eqref{eq:QCQP} is a \emph{regularly defined smooth manifold}, then, for almost all cost matrices $C$,
%% PAA ADD begin
if $Y$ satisfies first- and second-order necessary optimality conditions for~\eqref{eq:QCQP}, then $Y$ is a global optimum of~\eqref{eq:QCQP} and, since $\frac{p(p+1)}{2} \geq m$, $X = YY\transpose$ is a global optimum of~\eqref{eq:SDP}. In other words,
%% PAA ADD end
first- and second-order necessary optimality conditions for~\eqref{eq:QCQP} are also \emph{sufficient} for global optimality---an unusual theoretical guarantee in non-convex optimization.

Notice that this is a statement about the optimization problem itself, not about specific algorithms. Interestingly, known algorithms for optimization on manifolds converge to \emph{second-order critical points},\footnote{Second-order critical points satisfy first- and second-order necessary optimality conditions.} regardless of initialization~\citep{boumal2016globalrates}.

For the specified class of SDPs, our result improves on those of~\citep{burer2005local} in two important ways. Firstly, for almost all $C$, we formally exclude the existence of spurious local optima.\footnote{Before Prop.\ 2.3 in~\citep{burer2005local}, the authors write: ``The change of variables $X = YY\transpose$ does not introduce any extraneous local minima.'' This is sometimes misunderstood to mean~\eqref{eq:QCQP} does not have spurious local optima, when it actually means that the local optima of~\eqref{eq:QCQP} are in exact correspondence with the local optima of ``\eqref{eq:SDP} \emph{with the extra constraint $\rank(X) \leq p$},'' which is also non-convex and thus also liable to having local optima. Unfortunately, this misinterpretation has led to some confusion in the literature.} Secondly, we only require the computation of second-order critical points of~\eqref{eq:QCQP} rather than local optima (which is hard in general~\citep{vavasis1991nonlinear}). Below, we make a statement about computational complexity, and we illustrate the practical efficiency of the proposed methods through numerical experiments.
%, whereas existing results only contained those to peculiar faces of the search space of the SDP, without statement as to the difficulties they may pose for local optimization methods.

SDPs which satisfy the compactness and smoothness assumptions occur in a number of applications including Max-Cut, robust PCA,  $\mathbb{Z}_2$-synchronization, community detection, cut-norm approximation, phase synchronization, phase retrieval, synchronization of rotations and the trust-region subproblem---see Section~\ref{sec:examples} for references.

%\TODO{QBP : \url{http://www.maths.lth.se/computers/vision/publdb/reports/pdf/olsson-eriksson-etal-cvpr-07.pdf} ; numerical example should be robust pca (McCoy and Tropp, Algorithm 1 (not 2)) / phasecut if get good results.}

%\TODO{Take inspiration from complexity paper to sell this. Perhaps just need to open with a more general paragraph, and lead to max-cut clearly as ``just an example''.}

%\TODO{Avoid a glitch with matrix $A$ and $C$.}

%\TODO{We should add a note about what was known by BM about
%	local optima of high rank having to be in weird faces, we should
%	highlight this A LOT because the reviewers will think that BM had
%	already proven this and will quickly dismiss the paper. So this
%	distinction should be done in the paper very often, I would say at
%	least 3 or 4 times.}

\subsection*{A simple example: the Max-Cut problem}

Given an undirected graph, Max-Cut is the NP-hard problem of clustering the~$n$ nodes of this graph in two classes, $+1$ and $-1$, such that as many edges as possible join nodes of different signs. If $C$ is the adjacency matrix of the graph, Max-Cut is expressed as
\begin{align}
	\max_{x\in\Rn} \frac{1}{4}\sum_{i,j=1}^n C_{ij}(1-x_ix_j) \quad \textrm{ s.t. } \quad x_1^2 = \cdots = x_n^2 = 1.
	\tag{Max-Cut}
	\label{eq:maxcut}
\end{align}
%Introducing the rank 1, positive semidefinite matrix $X=xx\transpose$, this is equivalent to %\TODO{notation}
%\begin{align*}
%	\max_{X\in\Snn} \frac{1}{4}\left( \1\transpose A \1 - \trace(AX) \right) \textrm{ s.t. } \diag(X) = \1, X \succeq 0, \rank(X) = 1.
%\end{align*}

%The Goemans-Williamson relaxation for this problem is a semidefinite program (SDP)
Introducing the positive semidefinite matrix $X = xx\transpose$, both the cost and the constraints may be expressed linearly in terms of $X$. Ignoring that $X$ has rank 1 yields the well-known convex relaxation in the form of a semidefinite program (up to an affine transformation of the cost): %which consists in dropping the rank constraint. Up to constants in the cost, this yields
\begin{align}
	\min_{X\in\Snn} \inner{C}{X} \quad \textrm{ s.t. } \quad \diag(X) = \1, \ X \succeq 0.
	\tag{Max-Cut SDP}
	\label{eq:maxcutsdp}
\end{align}
%which can be derived by rewriting \label{eq:maxcut} in terms of $xx\transpose$, and using $X \succeq 0$ as a proxy for $X = xx\transpose$, which can be viewed as dropping a rank constraint on $X$. 
%for which the celebrated result of \cite{} gives an approximation ratio of .878. 
If a solution $X$ of this SDP has rank 1, then $X=xx\transpose$ for some $x$ which is then an optimal cut. In the general case of higher rank $X$, \citet{goemans1995maxcut} exhibited the celebrated rounding scheme to produce approximately optimal cuts (within a ratio of .878) from $X$.

The corresponding Burer--Monteiro non-convex problem with rank bounded by $p$ is:
\begin{align}
	\min_{Y\in\Rnp} \inner{CY}{Y} \quad \textrm{ s.t. } \quad  \diag(YY\transpose) = \1.
	\tag{Max-Cut BM}
	\label{eq:maxcutBM}
\end{align}
%For $p \ll n$, it is much lower dimensional than the SDP. On the other hand, it is non-convex, which means there is no reason a priori to suspect it can be solved to global optimality. Nevertheless, \citet{burer2005local} make the following encouraging points (which will be justified below):
%\begin{enumerate}
%%i)
%	\item If $Y$ is a \emph{rank-deficient local optimum} for~\eqref{eq:maxcutBM}, then $X=YY\transpose$ is globally optimal for~\eqref{eq:maxcutsdp}; and
%%ii)
%	\item In practice, L-BFGS on~\eqref{eq:maxcutBM} in augmented Lagrangian form with $\frac{p(p+1)}{2} > n$ (a local method)  typically converges to a rank-deficient local optimum, and hence, a global optimum.
%\end{enumerate}

%\begin{enumerate}
%%	\item If $\frac{p(p+1)}{2} \geq n$, then any optimal $Y$ for~\eqref{eq:maxcutBM} turns into an optimal $X=YY\transpose$ for~\eqref{eq:maxcutsdp}, that is, the two problems are equivalent.
%	\item If $Y$ is a \emph{rank-deficient local optimizer} for~\eqref{eq:maxcutBM}, then $X=YY\transpose$ is globally optimal for~\eqref{eq:maxcutsdp}.
%	\item In practice, running L-BFGS on~\eqref{eq:maxcutBM} in augmented Lagrangian form
%	% I checked the above denomination of the algorithm they use; it's fine; their sentence is: "limited memory BFGS, first-order augmented Lagrangian algorithm"
%	with $\frac{p(p+1)}{2} > n$ appears to typically converge to a rank-deficient local optimizer, hence, a global optimizer.
%\end{enumerate}

The constraint $\diag(YY\transpose)=\1$ requires each row of $Y$ to have unit norm; that is: $Y$ is a point on the Cartesian product of $n$ unit spheres in $\Rp$, which is a smooth manifold. Furthermore, all $X$ feasible for the SDP have identical trace equal to $n$, so that the search space of the SDP is compact. Thus, our results stated below apply:

\begin{quote}
	\emph{For $p = \left\lceil\sqrt{2n}\,\right\rceil$, for almost all $C$, even though~\eqref{eq:maxcutBM} is non-convex,
	%it has no inescapable traps.
	any local optimum $Y$ is a global optimum (and so is $X=YY\transpose$), and all saddle points have an escape (the Hessian has a negative eigenvalue).
	}
\end{quote}
We note that, for $p>n/2$, the same holds for \emph{all} $C$~\citep{boumal2015staircase}.

\subsection*{Notation}

$\Snn$ is the set of real, symmetric matrices of size $n$. A symmetric matrix $X$ is positive semidefinite ($X \succeq 0$) if and only if $u\transpose X u \geq 0$ for all $u\in\Rn$. For matrices $A, B$, the standard Euclidean inner product is $\inner{A}{B} = \trace(A\transpose B)$. The associated (Frobenius) norm is $\|A\| = \sqrt{\inner{A}{A}}$. $\Id$ is the identity operator and $I_n$ is the identity matrix of size $n$.

\section{Main results}

%We consider semidefinite programs of the form
%\begin{align}
%	f^* = \min_{X\in\Snn} \inner{C}{X} \textrm{ subject to } \calA(X) = b, X \succeq 0,
%\tag{SDP}
%%\label{eq:SDP}
%\end{align}
%where $\calA \colon \Snn \to \Rm$ is a linear operator capturing $m$ equality constraints.
%%
%Restricting the search space of~\eqref{eq:SDP} to matrices of rank at most $p$ (and reparameterizing the search space by factorizing $X = YY\transpose$) yields a quadratically constrained quadratic program,
%\begin{align}
%	q^* = \min_{Y\in\Rnp} \inner{CY}{Y} \textrm{ subject to } \calA(YY\transpose) = b.
%\tag{QCQP}
%%\label{eq:QCQP}
%\end{align}
%Since~\eqref{eq:SDP} is a relaxation of~\eqref{eq:QCQP} (up to parameterization), $q^* \geq f^*$. In general, \eqref{eq:QCQP} is non-convex, making it a priori unclear how to solve it globally. The benefits are that it is lower dimensional than~\eqref{eq:SDP} and has no conic constraint.

%\emph{Second-order critical points} are points which satisfy first- and second-order necessary optimality conditions.
%Our main result shows an equivalence between the global optima of a class of~\eqref{eq:SDP}'s and the second-order critical points of the corresponding~\eqref{eq:QCQP}'s, an unusual property in non-convex programming.
%Importantly, all local optimizers are second-order critical points.
Our main result establishes conditions under which first- and second-order necessary optimality conditions for~\eqref{eq:QCQP} are sufficient for global optimality. Under those conditions, it is a fortiori true that global optima of~\eqref{eq:QCQP} map to global optima of~\eqref{eq:SDP}, so that local optimization methods on~\eqref{eq:QCQP} can be used to solve the higher-dimensional, cone-constrained~\eqref{eq:SDP}.

We now specify the necessary optimality conditions of~\eqref{eq:QCQP}. Under the assumptions of our main result below (Theorem~\ref{thm:masterthm}), the search space
\begin{align}
\calM & = \calM_p = \{ Y \in \Rnp : \calA(YY\transpose) = b \}
\label{eq:M}
\end{align}
is a smooth and compact manifold of dimension $np - m$. As such, it can be linearized at each point $Y\in\calM$ by a tangent space, differentiating the constraints~\citep[eq.\,(3.19)]{AMS08}:
\begin{align}
	\T_Y\calM & = \{ \dot Y \in \Rnp : \calA(\dot Y Y\transpose + Y \dot Y\transpose) = 0 \}.
	\label{eq:TYM}
\end{align}
Endowing the tangent spaces of $\calM$ with the (restricted) Euclidean metric $\inner{A}{B} = \trace(A\transpose B)$ turns $\calM$ into a Riemannian submanifold of $\Rnp$.
In general, second-order optimality conditions can be intricate to handle~\citep{ruszczynski2006nonlinear}. Fortunately, here, the smoothness of both the search space~\eqref{eq:M} and the cost function
\begin{align}
	f(Y) & = \inner{CY}{Y}
	\label{eq:f}
\end{align}
make for straightforward conditions. In spirit, they coincide with the well-known conditions for unconstrained optimization.
%(Appendix~\ref{sec:proofs} gives definitions and explicit expressions for the Riemannian gradient and Hessian.) %; see Section~\ref{sec:proofs} for definitions and formulas.
As further detailed in Appendix~\ref{sec:proofs}, the Riemannian gradient $\grad f(Y)$ is the orthogonal projection of the classical gradient of $f$ to the tangent space $\T_Y\calM$. The Riemannian Hessian of $f$ at $Y$ is a similarly restricted version of the classical Hessian of $f$ to the tangent space.
%Here, the \emph{Riemannian gradient} is the orthogonal projection of the classical gradient to the tangent space, and the \emph{Riemannian Hessian} is the differential of the gradient vector field on $\calM$---see Section~\ref{sec:proofs}.
\begin{definition}\label{def:criticalpts}
	A (first-order) \emph{critical point} for~\eqref{eq:QCQP} is a point $Y \in \calM$ such that
	\begin{align}
	\grad f(Y) & = 0,
	\tag{1st order nec.\ opt.\ cond.}
	\end{align}
	where $\grad f(Y) \in \T_Y\calM$ is the Riemannian gradient at $Y$ of
	%$f|_\calM$.
	$f$ restricted to $\calM$.
	A \emph{second-order critical point} for~\eqref{eq:QCQP} is a critical point $Y$ such that
	\begin{align}
	\Hess f(Y) \succeq 0,
	\tag{2nd order nec.\ opt.\ cond.}
	\end{align}
	where $\Hess f(Y) \colon \T_Y\calM \to \T_Y\calM$ is the Riemannian Hessian at $Y$ of $f$ restricted to $\calM$ (a symmetric linear operator).
\end{definition}
\begin{proposition}\label{prop:optconditions}
	All local (and global) optima of~\eqref{eq:QCQP} are second-order critical points.
\end{proposition}
\begin{proof}
	See~\cite[Rem.~4.2 and Cor.~4.2]{yang2012optimality}.
	%\TODO{ref; perhaps in the complexity paper.}
\end{proof}
%
%\emph{First-order critical points} (or simply critical points) satisfy the first-order necessary optimality conditions, namely, $\grad f(Y) = 0$. \emph{Second-order critical points} further satisfy the second-order necessary optimality conditions, namely, $\Hess f(Y) \succeq 0$, where $\Hess f(Y)$ is a symmetric operator on $\T_Y\calM$. These are the standard first- and second-order KKT conditions for~\eqref{eq:QCQP}: all local optimizers are second-order critical points.\footnote{This statement normally requires to check that certain constraint qualifications hold; this is indeed the case. \TODO{ref some things}} Theorem~\ref{thm:masterthm} states that these conditions are also sufficient for global optimality.
%

We can now state our main result. In the theorem statement below,
%``for almost all $C$'' means that the matrices $C$ for which the property does not hold (if they exist) form a (Lebesgue) zero-measure subset of $\Snn$.
%Here,
``for almost all $C$'' means potentially troublesome cost matrices form at most a (Lebesgue) zero-measure subset of $\Snn$, in the same way that almost all square matrices are invertible.
In particular, given any matrix $C\in\Snn$, perturbing $C$ to $C+\sigma W$ where $W$ is a Wigner random matrix results in an acceptable cost matrix with probability 1, for arbitrarily small $\sigma > 0$.
\begin{theorem}\label{thm:masterthm}
	Given constraints $\calA \colon \Snn \to \Rm$, $b\in\Rm$ and $p$ satisfying $\frac{p(p+1)}{2} > m$, if
%	(i) the search space of~\eqref{eq:SDP} is compact and (ii) the search space of~\eqref{eq:QCQP} is a smooth manifold, 
	\begin{itemize}
		\item[(i)] the search space of~\eqref{eq:SDP} is compact; and
		\item[(ii)] the search space of~\eqref{eq:QCQP} is a regularly-defined smooth manifold, in the sense that $A_1Y, \ldots, A_mY$ are linearly independent in $\Rnp$ for all $Y \in \calM$ (see Appendix~\ref{apdx:regularity}),
	\end{itemize}
	then for almost all cost matrices $C\in\Snn$, any second-order critical point of~\eqref{eq:QCQP} is globally optimal. Under these conditions, if $Y$ is globally optimal for~\eqref{eq:QCQP}, then the matrix $X = YY\transpose$ is globally optimal for~\eqref{eq:SDP}. %In particular, given an optimal $Y$ for , the matrix $X = YY\transpose$ is optimal for~\eqref{eq:SDP}.
\end{theorem}
The assumptions are discussed in the next section. The proof---see Appendix~\ref{sec:proofs}---follows directly from the combination of two intermediate results:
\begin{enumerate}
	\item If $Y$ is \emph{rank deficient} and second-order critical for~\eqref{eq:QCQP}, then it is globally optimal and $X=YY\transpose$ is optimal for~\eqref{eq:SDP}; and
	\item If $\frac{p(p+1)}{2} > m$, then, for almost all $C$, every first-order critical $Y$ is rank-deficient.
\end{enumerate}
The first step holds in a more general context,
% I was going to write: without the manifold assumption; but actually in there they also need to assume constraint qualifications, which imply the search space is locally a manifold. So it's not clear they proved a much more general result. But it is true that in the manifold case you can also make the statement for nonlinear convex functions, which Journée et al. did for a restricted set of manifolds.
as previously established by~\citet{sdplr,burer2005local}. The second step is new and crucial, as it allows to formally exclude the existence of spurious local optima, generically in $C$, thus resolving the caveat mentioned in the introduction.

%\TODO{copied--}Note that, at the cost of restricting the class of SDP's it applies to, this theorem improves on the results of~\citep{burer2005local} in two important ways. Firstly, it only requires the computation of second-order critical points rather than local optima. Secondly, in relation to the quote in the introduction, for almost any given $C$, we formally exclude the existence of spurious local optima, whereas existing results only contained those to peculiar faces of the search space of the SDP, without statement as to the difficulties they may pose for local optimization methods.

%\begin{proof}
%	Combine Corollary~\ref{cor:rankdefsocpopt} and Lemma~\ref{lem:criticalptsrankdeficient} in Section~\ref{sec:proofs}.
%\end{proof}
%In the following sections, we subsequently justify the assumptions of this theorem, illustrate it with applications, and prove it.

%We now specify the necessary optimality conditions more precisely and discuss computational issues. 

The smooth structure of~\eqref{eq:QCQP} naturally suggests using \emph{Riemannian optimization} to solve it~\citep{AMS08}, which is something that was already proposed by~\citet{journee2010low} in the same context.
Importantly, known algorithms
%---in particular, the \emph{Riemannian trust-region method} (RTR)---
converge to second-order critical points regardless of initialization.
%, as per recent results in Riemannian optimization~\citep{boumal2016globalrates}.
%%%This is in stark contrast with previous analyses
%%%%of~\eqref{eq:QCQP} in relation to~\eqref{eq:SDP}
%%%which required computation of local optima~\citep{sdplr,journee2010low}, which is hard in general~\citep{vavasis1991nonlinear}.
We state here a recent computational result to that effect.
%We note that, for standard trust-region algorithms, the rate of $1/\varepsilon^3$ is sharp even when optimizing over a linear space~\citep{cartis2012complexity}.
\begin{proposition}\label{prop:complexity}
	Under the numbered assumptions of Theorem~\ref{thm:masterthm}, the Riemannian trust-region method (RTR)~\citep{genrtr} initialized with any $Y_0\in\calM$ returns in $\calO(1/\varepsilon_g^2\varepsilon_H^{} + 1/\varepsilon_H^3)$ iterations a point $Y\in\calM$ such that
	\begin{align*}
		f(Y) & \leq f(Y_0), & \|\grad f(Y)\| & \leq \varepsilon_g, & \textrm{ and } & & \Hess f(Y) \succeq -\varepsilon_H \Id.
	\end{align*}
\end{proposition}
\begin{proof}
	Apply the main results of~\citep{boumal2016globalrates} using that $f$ has locally Lipschitz continuous gradient and Hessian in $\Rnp$ and $\calM$ is a compact submanifold of $\Rnp$.
	%One must use the Riemannian Hessian in the trust-region model and
%	The algorithm requires a \emph{second-order retraction}: a map $R_Y$ from $\T_Y\calM$ to $\calM$ which agrees with geodesics up to second order around $Y$ (and smoothly depends on $Y$ and its input). By~\citep[Thm.~22]{absil2012retractions}, given a point $Y$ and a tangent vector $\dot Y$, setting $R_Y(\dot Y)$ to be the orthogonal projection of $Y+\dot Y$ to $\calM$ is acceptable.
%	\TODO{Will reference a main theorem and corollaries in paper with Absil and Cartis to check that assumptions are fulfilled. Don't go into details of cost of iterations... Need a second order retraction: consider orthogonal projection and~\citep{absil2012retractions}. Need to make sure that projection \emph{and} retraction will take poly time. Otherwise, say each iteration requires this and that and such.}
\end{proof}
Essentially, each iteration of RTR requires evaluation of one cost and one gradient, a bounded number of Hessian-vector applications, and one projection from $\Rnp$ to $\calM$. In many important cases, this projection amounts to Gram--Schmidt orthogonalization of small blocks of $Y$---see Section~\ref{sec:examples}. % about the Hessian: worst case, you apply it \dim\calM times to get the whole matrix in some basis, then apply the polytime ideas described in the complexity paper. no dependence on \varepsilon_H.

Proposition~\ref{prop:complexity} bounds worst-case iteration counts for arbitrary initialization. In practice, a good initialization point may be available, making the local convergence rate of RTR more informative. For RTR, one may expect superlinear or even quadratic local convergence rates near isolated local minimizers~\citep{genrtr}.
%\footnote{Minimizers are never isolated in~\eqref{eq:QCQP} because of the invariance in the cost: $f(YQ) = f(Y)$ for any orthogonal $Q \in O(p)$. Notwithstanding, superlinear rates are observed in practice. This is partly addressed in~\citep{journee2010low} where optimization is conducted on the quotient manifold $\calM / O(p)$.}
While minimizers are not isolated in our case~\citep{journee2010low}, experiments show a characteristically superlinear local convergence rate in practice~\citep{boumal2015staircase}.
This means high accuracy solutions can be achieved, as demonstrated in Appendix~\ref{sec:XP}. %, contrary to what one would expect if using first-order methods for SDP's.

Thus, under the conditions of Theorem~\ref{thm:masterthm}, generically in $C$, RTR converges to global optima. In practice, the algorithm returns after a finite number of steps, and only \emph{approximate} second-order criticality is guaranteed. Hence, it is interesting to bound the optimality gap in terms of the approximation quality. Unfortunately, we do not establish such a result for small $p$. Instead, we give an a posteriori computable optimality gap bound which holds for \emph{all} $p$ and for \emph{all} $C$.
%For specific applications, it may be possible to control the bound a priori for smaller $p$ as well.
In the following statement, the dependence of $\calM$ on $p$ is explicit, as $\calM_p$. The proof is in Appendix~\ref{sec:proofs}.

%, but instead give a posteriori computable bounds on the optimality gap.

%It is only ever possible to compute points which approximately satisfy optimality conditions. The result below quantifies the optimality gap for approximate second-order critical points which are also rank deficient. 
%\TODO{If can't get rid of rank deficiency assumption, comment. In particular, for ortho-cut, we have better. But as such, this doesn't do much for us... --- Let's be realistic: we won't get rid of it; it would be a far finer result than the dimensionality counting argument, and we have no lead at this point. i'm not even sure this theorem should be here; the general result in S is more practical, and just honestly say that you can't use that a priori, only a posteriori.}
%\begin{theorem} \TODO{delete}
%	For any rank-deficient $Y\in\calM$, if $\|\grad f(Y)\| \leq \varepsilon_g$ and $\Hess f(Y) \succeq -\varepsilon_H \Id$, the optimality gap with respect to~\eqref{eq:SDP} is bounded as
%	\begin{align}
%		0 \leq 2(f(Y) - f^*) \leq \varepsilon_H R + \varepsilon_g \sqrt{R},
%	\end{align}
%	where $R$ is the maximal trace of any $X$ feasible for~\eqref{eq:SDP}. If all such $X$'s have the same trace, then the bound simplifies to
%	\begin{align}
%	0 \leq f(Y) - f^* \leq \frac{\varepsilon_H R}{2}.
%	\end{align}
%\end{theorem}
\begin{theorem} \label{thm:approxtolerance} %\TODO{define $\calM_p$}
	 % I decided to not phrase the theorem in terms of rank deficient $Y$, because numerically it will never be exact, and these theorems are all about approximations.
	Let $R < \infty$ be the maximal trace of any $X$ feasible for~\eqref{eq:SDP}.
	For any $p$ such that $\calM_p$ and $\calM_{p+1}$ are smooth manifolds (even if $\frac{p(p+1)}{2} \leq m$) and for any $Y\in\calM_p$, form $\tilde Y = \left[ Y | 0_{n\times 1} \right]$ in $\calM_{p+1}$. The optimality gap at $Y$ is bounded as
	\begin{align}
%		-\lambdamin(\Hess f(\tilde Y)) \geq \frac{2(f(Y)-f^*) - \sqrt{R}\|\grad f(Y)\|}{R}.
		0 \leq 2(f(Y)-f^*) \leq \sqrt{R} \|\grad f(Y)\| - R \lambdamin(\Hess f(\tilde Y)).
		\label{eq:lambdaminbound}
	\end{align}
	%
	% If all feasible $X$'s have the same trace, then the bound simplifies to
	If all feasible $X$ have the same trace $R$ and there exists a positive definite feasible $X$, then the bound simplifies to
	\begin{align}
		0 \leq 2(f(Y)-f^*) \leq - R \lambdamin(\Hess f(\tilde Y))
	\end{align}
	so that $\|\grad f(Y)\|$ needs not be controlled explicitly.
%	These expressions also give a posteriori computable bounds on the optimality gap at approximately second-order critical points.
	If $p>n$, the bounds hold with $\tilde Y = Y$.
\end{theorem}
In particular, for $p=n+1$, the bound can be controlled a priori: approximate second-order critical points are approximately optimal, for any $C$.\footnote{With $p=n+1$, problem~\eqref{eq:QCQP} is no longer lower dimensional than~\eqref{eq:SDP}, but retains the advantage of not involving a positive semidefiniteness constraint.}
\begin{corollary}
	Under the assumptions of Theorem~\ref{thm:approxtolerance}, if $p=n+1$ and $Y\in\calM$ satisfies both $\|\grad f(Y)\| \leq \varepsilon_g$ and $\Hess f(Y) \succeq -\varepsilon_H \Id$, then $Y$ is approximately optimal in the sense that
	\begin{align*}
	0 \leq 2(f(Y)-f^*) \leq \sqrt{R} \varepsilon_g + R \varepsilon_H.
	\end{align*}
	Under the same condition as in Theorem~\ref{thm:approxtolerance}, the bound can be simplified to $R \varepsilon_H$.
\end{corollary}
This works well with Proposition~\ref{prop:complexity}. For any $p$, equation~\eqref{eq:lambdaminbound} also implies the following:
\begin{align*}
	\lambdamin(\Hess f(\tilde Y)) \leq - \frac{2(f(Y)-f^*) - \sqrt{R}\|\grad f(Y)\|}{R}.
\end{align*}
That is,
% if $Y$ is approximately critical but $YY\transpose$ is not approximately optimal for~\eqref{eq:SDP}, then the Hessian of $f$ at $\tilde Y$ has a strongly negative eigenvalue, allowing to escape the approximate saddle point.
for any $p$ and any $C$, an approximate critical point $Y$ in $\calM_p$ which is far from optimal maps to a comfortably-escapable approximate saddle point $\tilde Y$ in $\calM_{p+1}$.

This suggests an algorithm as follows. For a starting value of $p$ such that $\calM_p$ is a manifold, use RTR to compute an approximate second-order critical point $Y$. Then, form $\tilde Y$ in $\calM_{p+1}$ and test the left-most eigenvalue of $\Hess f(\tilde Y)$.\footnote{It may be more practical to test $\lambdamin(S)$~\eqref{eq:S} rather than $\lambdamin(\Hess f)$. Lemma~\ref{lem:fromHesstoS} relates the two. See~\citep[\S3.3]{journee2010low} to construct escape tangent vectors from $S$.} If it is close enough to zero, this provides a good bound on the optimality gap. If not, use an (approximate) eigenvector associated to $\lambdamin(\Hess f(\tilde Y))$ to escape the approximate saddle point and apply RTR from that new point in $\calM_{p+1}$; iterate. In the worst-case scenario, $p$ grows to $n+1$, at which point all approximate second-order critical points are approximate optima. Theorem~\ref{thm:masterthm} suggests $p = \left\lceil \sqrt{2m} \, \right\rceil$ should suffice for $C$ bounded away from a zero-measure set. Such an algorithm already features with less theory in~\citep{journee2010low} and~\citep{boumal2015staircase}; in the latter, it is called the \emph{Riemannian staircase}, for it lifts~\eqref{eq:QCQP} floor by floor.

\subsection*{Related work}

Low-rank approaches to solve SDPs have featured in a number of recent research papers. We highlight just two which illustrate different classes of SDPs of interest.

% https://arxiv.org/pdf/1605.09527v2.pdf
\citet{shah2016biconvex} tackle SDPs with linear cost and linear constraints (both equalities and inequalities) via low-rank factorizations, assuming the matrices appearing in the cost and constraints are positive semidefinite. They propose a non-trivial initial guess to partially overcome non-convexity with great empirical results, but do not provide optimality guarantees.

% http://akyrillidis.github.io/pubs/Conferences/FGD.pdf
\citet{bhojanapalli2016dropping} on the other hand consider the minimization of a convex cost function over positive semidefinite matrices, without constraints. Such problems could be obtained from generic SDPs by penalizing the constraints in a Lagrangian way. Here too, non-convexity is partially overcome via non-trivial initialization, with global optimality guarantees under some conditions.

Also of interest are recent results about the harmlessness of non-convexity in low-rank matrix completion~\citep{ge2016matrix,bhojanapalli2016global}. Similarly to the present work, the authors there show there is no need for special initialization despite non-convexity.

% This is an older one, but quite good
%\citep{kulis2007fast} http://www.jmlr.org/proceedings/papers/v2/kulis07a/kulis07a.pdf

% These are specifically for low-rank matrix completion
%\citep{ge2016matrix} https://arxiv.org/abs/1605.07272
%\citep{bhojanapalli2016global} https://arxiv.org/abs/1605.07221

% These are not for SDP's but for generic low-rank problems (rectangular matrices)
%\citep{mishra2011low}
%\citep{mishra2012fixed}
%\citep{uschmajew2015greedy}

\section{Discussion of the assumptions}

Our main result, Theorem~\ref{thm:masterthm}, comes with geometric assumptions on the search spaces of both~\eqref{eq:SDP} and~\eqref{eq:QCQP}
%---each working with a side assumption---
which we now discuss. Examples of SDPs which fit the assumptions of Theorem~\ref{thm:masterthm} are featured in the next section.

The assumption that the search space of~\eqref{eq:SDP}, 
\begin{align}
	\calC = \{ X \in \Snn : \calA(X) = b, X \succeq 0 \},
	\label{eq:C}
\end{align}
is \emph{compact} works in pair with the assumption $\frac{p(p+1)}{2} > m$ as follows. For~\eqref{eq:QCQP} to reveal the global optima of~\eqref{eq:SDP}, it is necessary that~\eqref{eq:SDP} admits a solution of rank at most~$p$. One way to ensure this is via the Pataki--Barvinok theorems~\citep{pataki1998rank,barvinok1995problems}, which state that all \emph{extreme points} of $\calC$ have rank $r$ bounded as $\frac{r(r+1)}{2} \leq m$. Extreme points are faces of dimension zero (such as vertices for a cube). When optimizing a linear cost function $\inner{C}{X}$ over a compact convex set $\calC$, at least one extreme point is a global optimum~\cite[Cor.\,32.3.2]{rockafellar1997convex}---this is not true in general if $\calC$ is not compact. Thus, under the assumptions of Theorem~\ref{thm:masterthm}, there is a point $Y\in\calM$ such that $X = YY\transpose$ is an optimal extreme point of~\eqref{eq:SDP}; then, of course, $Y$ itself is optimal for~\eqref{eq:QCQP}.

In general, the Pataki--Barvinok bound is tight, in that there exist extreme points of rank up to that upper-bound (rounded down)---see for example~\citep{laurent1996facial} for the Max-Cut SDP and~\citep{boumal2015staircase} for the Orthogonal-Cut SDP. Let $C$ (the cost matrix) be the negative of such an extreme point. Then, the unique optimum of~\eqref{eq:SDP} is that extreme point, showing that $\frac{p(p+1)}{2} \geq m$ is necessary for~\eqref{eq:SDP} and~\eqref{eq:QCQP} to be equivalent for all $C$. We further require a strict inequality because our proof relies on properties of rank deficient $Y$'s in $\calM$.

%Of course, compactness also ensures~\eqref{eq:SDP} attains its finite optimal value for all $C$.

The assumption that $\calM$ (eq.~\eqref{eq:M}) is a \emph{regularly-defined smooth manifold} works in pair with the am\-bi\-tion that the result should hold for (almost) all cost matrices $C$. The starting point is that, for a given non-convex smooth optimization problem---even a quadratically constrained quadratic program---computing local optima is hard in general~\citep{vavasis1991nonlinear}. Thus, we wish to restrict our attention to efficiently computable points, such as points which satisfy first- and second-order KKT conditions for~\eqref{eq:QCQP}---see~\cite[\S2.2]{sdplr} and \cite[\S3]{ruszczynski2006nonlinear}. This only makes sense if global optima satisfy the latter, that is, if KKT conditions are necessary for optimality. A global optimum $Y$ necessarily satisfies KKT conditions if \emph{constraint qualifications} (CQs) hold at $Y$~\citep{ruszczynski2006nonlinear}. The standard CQs for equality constrained programs are Robinson's conditions or metric regularity (they are here equivalent). They read as follows, assuming $\calA(YY\transpose)_i = \inner{A_i}{YY\transpose}$ for some matrices $A_1, \ldots, A_m \in \Snn$:
\begin{align}
	\textrm{CQs hold at } Y \textrm{ if } A_1Y, \ldots, A_mY \textrm{ are linearly independent in } \Rnp.
	\label{eq:CQ}
\end{align}
Considering almost all $C$, global optima could, a priori, be almost anywhere in $\calM$.
%\footnote{\TODO{This is actually not that clear; $\calM$ is not convex, so some points may not ever be optimal; need to make the point that actually yes, any of them could be optimal: has to do with faces of SDP's being exposed. But it's delicate because we're only covering almost all $C$, so you have to make the statement for all $C$ really; anyway: it's delicate, it's not that important, keep the story simple.}}
To simplify, we require CQs to hold at all $Y$'s in $\calM$ rather than only at the (unknown) global optima.
Under this condition, the constraints are independent at each point and ensure $\calM$ is a smooth embedded submanifold of $\Rnp$ of codimension $m$~\citep[Prop.\,3.3.3]{AMS08}. Indeed, tangent vectors $\dot Y \in \T_Y\calM$~\eqref{eq:TYM} are exactly those vectors that satisfy $\innersmall{A_i Y}{\dot Y} = 0$: under CQs, the $A_iY$'s form a basis of the normal space to the manifold at $Y$. %See~\citep[Thm.~3.3]{andreani2010constantrank} for the reverse statement, namely, that if $\calM$ is a manifold, then second-order CQ's hold at all $Y$---this confirms that the theorems in this paper depend on the geometry of $\calM$, and not on the (partially arbitrary) choice of $A_i$'s to represent it.

%Once it is decided that $\calM$ must be a manifold, we can step away from the specific representation of it via the matrices $A_1, \ldots, A_m$ and reason about optimality conditions on the manifold directly. Adding redundant constraints (for example, duplicating $A_1$) would break the CQs, but not the manifold structure. Hence, stating Theorem~\ref{thm:masterthm} in terms of manifolds better captures the role of $\calM$ than stating it in terms of CQ's. See also~\citep[Thm.~3.3]{andreani2010constantrank} for a proof that requiring $\calM$ to be a manifold around $Y$ is a type of CQ.

Finally, we note that Theorem~\ref{thm:masterthm} only applies for \emph{almost} all $C$, rather than all $C$. To justify this restriction, if indeed it is justified, one should exhibit a matrix $C$ that leads to suboptimal second-order critical points while other assumptions are satisfied. We do not have such an example. We do observe that \eqref{eq:maxcutsdp} on cycles of certain even lengths
%\footnote{\TODO{You guys know Michel: should we keep this? Or simply thank him for useful discussions in the acks?---We thank Michel Goemans for suggesting we study those graphs in this context.}}
has a unique solution of rank 1, while the corresponding~\eqref{eq:maxcutBM} with $p=2$ has suboptimal local optima (strictly, if we quotient out symmetries). This at least suggests it is not enough, for generic $C$, to set $p$ just larger than the rank of the solutions of the SDP. (For those same examples, at $p=3$, we consistently observe convergence to global optima.)

%\clearpage
\section{Examples of smooth SDPs}\label{sec:examples}

The canonical examples of SDPs which satisfy the assumptions in Theorem~\ref{thm:masterthm} are those where the diagonal blocks of $X$ or their traces are fixed. We note that the algorithms and the theory continue to hold for complex matrices, where the set of Hermitian matrices of size $n$ is treated as a real vector space of dimension $n^2$ (instead of $\frac{n(n+1)}{2}$ in the real case) with inner product $\inner{H_1}{H_2} = \Re\left\{ \Trace(H_1^*H_2^{}) \right\}$, so that occurrences of $\frac{p(p+1)}{2}$ are replaced by $p^2$.

Certain concrete examples of SDPs include:
\begin{align}
	\min_X \inner{C}{X} & \textrm{ s.t. } \Trace(X) = 1, X\succeq 0; \tag{fixed trace} \\
	\min_X \inner{C}{X} & \textrm{ s.t. } \diag(X) = \mathbf{1}, X\succeq 0; \tag{fixed diagonal} \\
	\min_X \inner{C}{X} & \textrm{ s.t. } X_{ii} = I_d, X\succeq 0. \tag{fixed diagonal blocks}
\end{align}
Their rank-constrained counterparts read as follows (matrix norms are Frobenius norms):
\begin{align}
\min_{Y\colon n\times p} \inner{CY}{Y} & \textrm{ s.t. } \|Y\| = 1; \tag{sphere} \\
\min_{Y\colon n\times p} \inner{CY}{Y} & \textrm{ s.t. } Y\transpose = \begin{bmatrix} y_1 & \cdots & y_n \end{bmatrix} \textrm{ and } \|y_i\| = 1 \textrm{ for all } i; \tag{product of spheres} \\
\min_{Y\colon qd\times p} \inner{CY}{Y} & \textrm{ s.t. } Y\transpose = \begin{bmatrix} Y_1 & \cdots & Y_{q} \end{bmatrix} \textrm{ and } Y_i\transpose Y_i^{} = I_d \textrm{ for all } i. \tag{product of Stiefel}
\end{align}

The first example has only one constraint: the SDP always admits an optimal rank~1 solution, corresponding to an eigenvector associated to the left-most eigenvalue of $C$. This generalizes to the trust-region subproblem as well.

For the second example, in the real case, $p=1$ forces $y_i = \pm 1$, allowing to capture combinatorial problems such as Max-Cut~\citep{goemans1995maxcut}, $\mathbb{Z}_2$-synchronization~\citep{javanmard2015phase} and community detection in the stochastic block model~\citep{abbe2014exact,bandeira2016lowrankmaxcut}. The same SDP is central in a formulation of robust PCA~\citep{mccoy2011robustpca} and is used to approximate the cut-norm of a matrix~\citep{alon2006cutnorm}. Theorem~\ref{thm:masterthm} states that for almost all $C$, $p =  \left\lceil \sqrt{2n}\, \right\rceil$ is sufficient.
%(Graph coloring~\citep{charikar2002semidefinite} and sphere packing on the sphere admit the same constraints with a nonlinear cost. Our theorems do not apply for nonlinear costs, but the optimization algorithms can be run.)
In the complex case, $p=1$ forces $|y_i| = 1$, allowing to capture problems where phases must be recovered; in particular, phase synchronization~\citep{bandeira2014tightness,singer2010angular} and phase retrieval via Phase-Cut~\citep{waldspurger2012phase}. For almost all $C$, it is then sufficient to set $p = \left\lfloor \sqrt{n}+1 \right\rfloor$.

In the third example, $Y$ of size $n\times p$ is divided in $q$ slices of size $d\times p$, with $p \geq d$. Each slice has orthonormal rows. For $p=d$, the slices are orthogonal (or unitary) matrices, allowing to capture Orthogonal-Cut~\citep{bandeira2013approximating} and the related problems of synchronization of rotations~\citep{wang2012LUD} and permutations. Synchronization of rotations is an important step in simultaneous localization and mapping, for example.
%~\citep{boumal2015staircase} (the former features both a linear and a nonsmooth cost; the nonsmooth cost can be smoothed~\citep{boumal2015staircase}).
Here, it is sufficient for almost all $C$ to let $p = \left\lceil \sqrt{d(d+1)q} \, \right\rceil$.

SDPs with constraints that are combinations of the above examples can also have the smoothness property;
%\footnote{For example, the SDP relaxation of the trust-region subproblem $\min_{x\transpose x = \Delta^2} x\transpose Ax + 2b\transpose x + c$ is $\min_{Z\in{\mathbb{S}^{(n+1)\times (n+1)}}} \inner{Z}{\begin{bmatrix}
%		A & b \\ b\transpose & c
%		\end{bmatrix}} \textrm{ s.t. } \trace(Z_{1:n,1:n}) = \Delta^2, Z_{n+1,n+1} = 1, Z\succeq 0$: it has two valid constraints, thus always admits a solution of rank 1 and $p=2$ is sufficient (at least) for almost all $A,b,c$.}
the right-hand sides $1$ and $I_d$ can be replaced by any positive definite right-hand sides by a change of variables. Another simple rule to check is if the constraint matrices $A_1, \ldots, A_m \in \Snn$ such that $\calA(X)_i = \inner{A_i}{X}$ satisfy $A_iA_j = 0$ for all $i \neq j$ (note that this is stronger than requiring $\inner{A_i}{A_j} = 0$), see~\citep{journee2010low}.

\section{Conclusions}

%We considered a class of SDP's whose search spaces are compact and such that, when they are restricted to 

%We have provided theoretical guarantees for efficiently optimizing a large class of SDPs via non-convex re-parametrization. Specifically,

The Burer--Monteiro approach consists in replacing optimization of a linear function $\inner{C}{X}$ over the convex set $\{X\succeq 0 : \calA(X) = b\}$ with optimization of the quadratic function $\inner{CY}{Y}$ over the non-convex set $\{Y\in\Rnp : \calA(YY\transpose) = b\}$. It was previously known that, if the convex set is compact and $p$ satisfies $\frac{p(p+1)}{2} \geq m$ where $m$ is the number of constraints, then these two problems have the same global optimum. It was also known from~\citep{burer2005local} that spurious local optima $Y$, if they exist, must map to special faces of the compact convex set, but without statement as to the prevalence of such faces or the risk they pose for local optimization methods. In this paper we showed that, if the set of $X$'s is compact and the set of $Y$'s is a regularly-defined smooth manifold, and if $\frac{p(p+1)}{2} > m$, then for almost all $C$, the non-convexity of the problem in $Y$ is benign, in that all $Y$'s which satisfy second-order necessary optimality conditions are in fact globally optimal. %: if the SDP is ``smooth'', in general, there are no spurious local optima.
%
%We showed that applying the Burer--Monteiro heuristic to solve SDP's with underlying smoothness by exploring 
%
%is almost always warranted, provided the maximum rank $p$ considered satisfies $p(p+1)/2 < m$, where $m$ is the number of constraints. Here, almost always means that, for a given SDP covered by our assumptions, for almost any cost matrix, the rank-restricted non-convex problem has no spurious local optimizers or inescapable saddles.
%of smooth SDPs yields non-convex optimization problems of a benign nature, in that
%all the points in their domain satisfying second order optimality are in fact global optima and thus solve the original SDP.

We further reference the Riemannian trust-region method~\citep{genrtr} to solve the problem in $Y$, as it was recently guaranteed to converge from any starting point to a point which satisfies second-order optimality conditions, with global convergence rates~\citep{boumal2016globalrates}. 
%(rates are provided).
In addition, for $p=n+1$, we guarantee that approximate satisfaction of second-order conditions implies approximate global optimality. We note that the $1/\varepsilon^3$ convergence rate in our results may be pessimistic. 
%\TODO{Strengthen? If we have numerics, it's easy.---See for instance \citep{boumal2015staircase,bandeira2014tightness} for numerical experiments which show that the algorithm studied here is faster than interior point methods on certain problems which satisfy the assumptions of our theorems.}
Indeed, the numerical experiments clearly show that high accuracy solutions can be computed fast using optimization on manifolds, at least for certain applications.
%You can cite my staircase paper and the tightness paper with Afonso and Amit and say that in those papers there are numerical experiments using the algorithm we propose to solve sdp's that satisfy the assumptions of of our theorem. There are no comparisons in the tightness paper, but in the staircase paper there are comparisons with interior point methods and it works much better (on synchronization of rotations problems)

Addressing a broader class of SDPs, such as those with inequality constraints or equality constraints that may violate our smoothness assumptions, could perhaps be handled by penalizing those constraints in the objective in an augmented Lagrangian fashion. We also note that, algorithmically, the Riemannian trust-region method we use applies just as well to nonlinear costs in the SDP.
%This is observed to be effective in certain situations~\citep{boumal2015staircase}.
We believe that extending the theory presented here to broader classes of problems is a good direction for future work.

\section*{Acknowledgment}

VV was partially supported by the Office of Naval Research. ASB was supported by NSF Grant DMS-1317308. Part
of this work was done while ASB was with the Department of
Mathematics at the Massachusetts Institute of Technology. We thank Wotao Yin and Michel Goemans for helpful discussions.

%\TODO{Wotao Yin; Philippe's grant; everything else}

% It is authorized to use small font for bibliography at NIPS 2016.
{\small
\bibliographystyle{plainnat}
\bibliography{../../../boumal}
}

\clearpage

\appendix

\section{Proofs and additional lemmas} \label{sec:proofs}

We start by working out explicit formulas for the Riemannian gradient and Hessian which appear in Definition~\ref{def:criticalpts}. Let $\Proj_Y \colon \Rnp \to \T_Y\calM$ be the orthogonal projector to the tangent space at $Y$ (eq.~\eqref{eq:TYM}), and let
\begin{align}
\nabla f(Y) & = 2CY, & \nabla^2 f(Y)[\dot Y] & = 2C\dot Y
\end{align}
be the (Euclidean) gradient and Hessian of the cost function~\eqref{eq:f}. The Riemannian gradient and Hessian of $f$ on $\calM$ are related to these as follows~\citep[see][eqs\,(3.37),~(5.15)]{AMS08}:
\begin{align}
\grad f(Y) & = \Proj_Y \nabla f(Y), \\ % = 2 \Proj_Y(CY), \label{eq:Riemanniangradient} \\
\forall \dot Y \in \T_Y\calM, \quad \Hess f(Y)[\dot Y] & = \Proj_Y \D\left( Y \mapsto \grad f(Y) \right)(Y)[\dot Y].
\label{eq:RiemannianHessian}
\end{align}
Let us focus on the gradient first. Since $\grad f(Y)$ is a tangent vector at $Y$~\eqref{eq:TYM},\footnote{For non-symmetric $B\in\Rnn$, note that $\calA(B) = \calA\big( \frac{B+B\transpose}{2} \big)$.}
\begin{align}
\calA(\grad f(Y)Y\transpose) & = 0,
\end{align}
and since it is the orthogonal projection of $\nabla f(Y)$ to the tangent space, there exists $\mu \in \Rm$ such that
\begin{align}
\grad f(Y) + 2 \calA^*(\mu)Y & = \nabla f(Y) = 2CY,
\label{eq:gradRandE}
\end{align}
where $\calA^* \colon \Rm \to \Snn$ is the adjoint of $\calA$. Indeed, considering symmetric matrices $A_1, \ldots, A_m$ such that $\calA(X)_i = \inner{A_i}{X}$, matrices $\calA^*(\mu)Y = \mu_1 A_1Y + \cdots + \mu_m A_mY$ span the normal space to the manifold at $Y$. Right-multiply~\eqref{eq:gradRandE} with $Y\transpose$ and apply $\calA$ to obtain
\begin{align}
\calA\left(\calA^*(\mu)YY\transpose\right) & = \calA(CYY\transpose).
\label{eq:AstarmuY}
\end{align}
Under the assumption that the $A_iY$'s are linearly independent, $\mu$ is the unique solution to this linear system---for KKT points, these are the Lagrange multipliers.
%In general, $\mu$ may not be unique but $\calA^*(\mu)Y$ always is, since~\eqref{eq:AstarmuY} is the result of an orthogonal projection.
Furthermore, contrary to classical KKT conditions, $\mu$ is defined for \emph{all} feasible $Y$ (not only for KKT points) and can be found by solving~\eqref{eq:AstarmuY}.\footnote{For the Max-Cut SDP for example, $\calA = \diag$ and $\mu = \diag(CYY\transpose)$.}
%If $\mu$ is not unique, for ease of exposition, let $\mu = \mu(Y) = \mu(YY\transpose)$ be the smallest-norm acceptable $\mu$
This $\mu$ is a well-defined, differentiable function of $Y$.\footnote{Eq.~\eqref{eq:AstarmuY} is equivalent to $G\mu = \calA(CYY\transpose)$, where $G_{ij} = \inner{A_iY}{A_jY}$. For all $Y\in\calM$, $G$
%	has constant rank equal to the codimension of $\calM$.
	is invertible since $A_1Y, \ldots, A_mY$ are linearly independent.
	Hence, $\mu = G^{-1} \calA(CYY\transpose)$ is differentiable in $Y$ at $Y\in\calM$.}
Using this definition of $\mu$, let
\begin{align}
	S & = S(Y) = S(YY\transpose) = C - \calA^*(\mu).
\label{eq:S}
\end{align}
First-order critical points then satisfy (using~\eqref{eq:gradRandE}):
\begin{align}
\frac{1}{2}\grad f(Y) = SY = 0.
\label{eq:firstorder}
\end{align}
We note in passing that $\mu(Y)$ is feasible for the dual of~\eqref{eq:SDP} exactly when $S(Y) \succeq 0$:
\begin{align}
d^* = \max_{\mu\in\Rm} b\transpose \mu \textrm{ subject to } C - \calA^*(\mu) \succeq 0,
\tag{DSDP}
\label{eq:DSDP}
\end{align}
which illustrates the importance of $S$ as a dual certificate for~\eqref{eq:SDP}.

Now let us turn to the Hessian of $f$. Equation~\eqref{eq:RiemannianHessian} requires computation of the differential of $\grad f(Y)$, which is
\begin{align*}
\D\big( Y \mapsto \grad f(Y) \big)(Y)[\dot Y] & = \D\big( Y \mapsto 2SY \big)(Y)[\dot Y] = 2S \dot Y + 2\dot SY,
\end{align*}
where $\dot S \triangleq \D S(Y)[\dot Y]$ is a symmetric matrix. Because of eq.~\eqref{eq:S}, $\dot S = \calA^*(\nu)$ for some $\nu \in \Rm$. Hence, for any tangent vector $\dot Z\in\T_Y\calM$~\eqref{eq:TYM}, we have $\innersmall{\dot Z}{\dot S Y} = \innersmall{\dot Z Y\transpose}{\calA^*(\nu)} = \innersmall{\calA(\dot Z Y\transpose)}{\nu} = 0$: $\dot S Y$ is orthogonal to the tangent space at $Y$. Using~\eqref{eq:RiemannianHessian}, we find that
\begin{align}
\frac{1}{2}\Hess f(Y)[\dot Y] = \Proj_Y S\dot Y.
\label{eq:secondorderfull}
\end{align}
The second-order condition for $Y$ is that $\Hess f(Y)$ be positive semidefinite on $\T_Y\calM$. Using that $\Proj_Y$ is a self-adjoint operator, it follows that this condition is equivalent to:
\begin{align}
\forall \dot Y \in \T_Y\calM, \quad \frac{1}{2} \innersmall{\dot Y}{\Hess f(Y)[\dot Y]}
% = \innersmall{\dot Y}{S\dot Y + \dot SY}
= \innersmall{\dot Y}{S\dot Y} \geq 0.
\label{eq:secondorder}
\end{align}
%Indeed, $\dot S = \calA^*(\nu)$ for some $\nu \in \Rm$, hence, $\innersmall{\dot Y}{\dot S Y} = \innersmall{\dot Y Y\transpose}{\calA^*(\nu)} = \innersmall{\calA(\dot Y Y\transpose)}{\nu} = 0$, since $\dot Y \in \T_Y\calM$~\eqref{eq:TYM}.

We now show that \emph{rank-deficient} second-order critical points are globally optimal. We obtain this result as a corollary to a more informative statement about optimality gap at approximately second-order critical points (assuming exact rank deficiency). The lemmas also show how $S$ can be used to control the optimality gap at approximate critical points without requiring rank deficiency.
%\TODO{Algorithmic use for this lemma: reach approx crit point; it's not rank deficient (in general). Add a column of zeros: still feasible, still approx critical, now rank deficient; check out lambdamin of Hessian (not necessarily small) : if it's small, the lemma gives an optimality bound; if it's not, a certificate vector gives an escape direction to leave the approximate saddle. Forget about algorithms: think about the landscape: for any given $p$, an approx socp which is far from optimal must map to a comfortably escapable saddle in $p+1$.}
%\TODO{For all results in this section, it is understood that $\calC$ is compact and $\calM$ is a manifold.}
This is valid for any $p$ and any $C$.
\begin{lemma} \label{lem:optimgap}
	For any $Y$ on the manifold $\calM$, if $\|\grad f(Y)\| \leq \varepsilon_g$ and $S(Y) \succeq -\frac{\varepsilon_H}{2} I_n$, then the optimality gap at $Y$ with respect to~\eqref{eq:SDP} is bounded as
	\begin{align}
	0 \leq 2(f(Y) - f^*) \leq \varepsilon_H R + \varepsilon_g \sqrt{R},
	\label{eq:lemoptimgap}
	\end{align}
	where $R = \max_{X\in\calC} \trace(X) < \infty$ measures the size of the compact set $\calC$~\eqref{eq:C}.
	If $I_n\in\im(\calA^*)$,
	% It's equivalent, at least I think so if Slater's condition holds; see notes March 18, 2016.
	the right hand side of~\eqref{eq:lemoptimgap} simplifies to $\varepsilon_H R$. This holds in particular if all $X\in\calC$ have same trace and $\calC$ has a relative interior point (Slater condition).
\end{lemma}
%\begin{proof}
%See appendix in supplementary materials.
%\end{proof}
\begin{proof}
	By assumption on $S(Y)$ (eq.~\eqref{eq:S}),
	\begin{align*}
	\forall \tilde X\in\calC, \quad -\frac{\varepsilon_H}{2} \trace(\tilde X) \leq \innersmall{S(Y)}{\tilde X} & = \innersmall{C}{\tilde X} - \innersmall{\calA^*(\mu(Y))}{\tilde X} = \innersmall{C}{\tilde X} - \innersmall{\mu(Y)}{b}.
	\end{align*}
	This holds in particular for $\tilde X$ optimal for~\eqref{eq:SDP}. Thus, we may set $\innersmall{C}{\tilde X} = f^*$; and certainly, $\trace(\tilde X) \leq R$. Furthermore,
	\begin{align*}
	\innersmall{\mu(Y)}{b} = \innersmall{\mu(Y)}{\calA(YY\transpose)} = \innersmall{C-S(Y)}{YY\transpose} = f(Y) - \innersmall{S(Y)Y}{Y}.
	\end{align*}
	Combining the typeset equations and using $\grad f(Y) = 2S(Y)Y$, we find
	\begin{align}
	0 \leq 2(f(Y) - f^*) \leq \varepsilon_H R + \innersmall{\grad f(Y)}{Y}.
	\label{eq:optgappartial}
	\end{align}
	In general, we do not assume $I_n \in \im(\calA^*)$ and we get the result by Cauchy--Schwarz on~\eqref{eq:optgappartial} and $\|Y\| = \sqrt{\trace(YY\transpose)} \leq \sqrt{R}$:
	\begin{align*}
	0 \leq 2(f(Y)-f^*) \leq \varepsilon_H R + \varepsilon_g \sqrt{R}.
	\end{align*}
	But if $I_n \in \im(\calA^*)$, then we show that $Y$ is a normal vector at $Y$, so that it is orthogonal to $\grad f(Y)$. Formally: there exists $\nu \in \Rm$ such that $I_n = \calA^*(\nu)$, and
	\begin{align*}
	\innersmall{\grad f(Y)}{Y} & = \innersmall{\grad f(Y) Y\transpose}{I_n} = \innersmall{\calA(\grad f(Y) Y\transpose)}{\nu} = 0,
	\end{align*}
	since $\grad f(Y) \in \T_Y\calM$~\eqref{eq:TYM}. This indeed allows to simplify~\eqref{eq:optgappartial}.
	
	To conclude, we show that if $\calC$ has a relative interior point $X'$ (that is, $\calA(X') = b$ and $X' \succ 0$) and if $\Trace(X)$ is a constant for all $X$ in $\calC$, then $I_n \in \im(\calA^*)$. Indeed, $\Snn = \im(\calA^*) \oplus \ker\calA$, so there exist $\nu\in\Rm$ and $M \in \ker \calA$ such that $I_n = \calA^*(\nu) + M$. Thus, for all $X$ in $\calC$,
	\begin{align*}
	0 = \trace(X - X') = \inner{\calA^*(\nu)+M}{X-X'} = \inner{M}{X-X'}.
	\end{align*}
	This implies that $M$ is orthogonal to all $X-X'$. These span $\ker\calA$ since $X'$ is interior. Indeed, for any $H\in\ker\calA$, since $X'\succ 0$, there exists $\varepsilon > 0$ such that $X \triangleq X'+\varepsilon H \succeq 0$ and $\calA(X) = b$, so that $X \in \calC$.
	%The other way around, it is clear that for all $X\in\calC$, $\calA(X-X') = 0$ so that $X-X' \in \ker \calA$. --- we don't need that part and it's trivial anyway.
	Hence, $M \in \ker\calA$ is orthogonal to $\ker \calA$. Consequently, $M = 0$ and $I_n = \calA^*(\nu)$.
\end{proof}
\begin{lemma}\label{lem:fromHesstoS}
	If $Y\in\calM$ is column rank deficient and $\Hess f(Y) \succeq -\varepsilon_H \Id$, then $S(Y) \succeq -\frac{\varepsilon_H}{2} I_n$.
\end{lemma}
\begin{proof}
	%	If $Y$ is not full column rank and $\Hess f(Y) \succeq -\varepsilon_H \Id$, then the condition on $S$ holds. Indeed,
	By assumption, there exists $z\in\Rp$, $\|z\| = 1$ such that $Yz = 0$. Thus, for any $x\in\Rn$, we can form $\dot Y = xz\transpose$: it is a tangent vector since $Y\dot Y\transpose = 0$~\eqref{eq:TYM}. Then, condition~\eqref{eq:secondorder} combined with the assumption on $\Hess f(Y)$ tells us
	\begin{align*}
	-\varepsilon_H \|x\|^2 \leq \innersmall{\dot Y}{\Hess f(Y)[\dot Y]} = 2\innersmall{\dot Y}{S\dot Y} = 2\innersmall{xz\transpose z x\transpose}{S} = 2 x\transpose S x.
	\end{align*}
	This holds for all $x\in\Rn$, hence $S \succeq -\frac{\varepsilon_H}{2} I_n$ as required.
\end{proof}
\begin{corollary} \label{cor:rankdefsocpopt}
	If $Y\in\calM_p$ is a column rank-deficient second-order critical point for~\eqref{eq:QCQP}, then it is optimal for~\eqref{eq:QCQP} and $X = YY\transpose$ is optimal for~\eqref{eq:SDP}. In particular, for $p>n$, all second-order critical points are optimal.
\end{corollary}
The first part of this corollary also appears as~\cite[Prop.\,4]{sdplr}, where the statement is made about local optima rather than second-order critical points.

At this point, we can already give a short proof of Theorem~\ref{thm:approxtolerance}.
\begin{proof}[Proof of Theorem~\ref{thm:approxtolerance}]
	%	\TODO{this is forward referencing things from section 5 ; move it there?}
	Since $\tilde Y \tilde Y\transpose = YY\transpose$, $S(\tilde Y) = S(Y)$; in particular, $f(\tilde Y) = f(Y)$ and $\|\grad f(\tilde Y)\| = \|\grad f(Y)\|$. Since $\tilde Y$ has deficient column rank, apply Lemmas~\ref{lem:optimgap} and~\ref{lem:fromHesstoS}. For $p>n$, there is no need to form $\tilde Y$ as $Y$ necessarily has deficient column rank.
\end{proof}

Based on Corollary~\ref{cor:rankdefsocpopt}, to establish Theorem~\ref{thm:masterthm} it is sufficient to show that, for almost all $C$, all second-order critical points are rank deficient already for small $p$. We show that in fact this is true even for first-order critical points. The argument is by dimensionality counting on $\Snn$: the set of all possible cost matrices $C$.
\begin{lemma} \label{lem:criticalptsrankdeficient}
	Under the assumptions of Theorem~\ref{thm:masterthm}, for almost all $C$, all critical points of~\eqref{eq:QCQP} are rank deficient.
\end{lemma}
%\begin{proof}
%See appendix in supplementary materials.
%\end{proof}
\begin{proof}
	Let $Y$ be a critical point for~\eqref{eq:QCQP}. By the first-order condition $S(Y)Y=0$~\eqref{eq:firstorder} and the definition of $S(Y) = C - \calA^*(\mu(Y))$~\eqref{eq:S}, there exists $\mu\in\Rm$ such that
	% http://math.stackexchange.com/a/259904/162453
	\begin{align}
	\rank Y \leq \nulll(C-\calA^*(\mu)) \leq \max_{\nu\in\Rm} \nulll(C-\calA^*(\nu)),
	\label{eq:baseinequality}
	\end{align}
	where $\nulll$ denotes the nullity (dimension of the kernel). This first step in the proof is inspired by~\cite[Thm.\,3]{wen2013orthogonality}.
	%\footnote{\TODO{Address why \cite[Thm.\,3]{wen2013orthogonality} is not satisfactory: as presented, limited to diagonal constraints ; no characterization of when the inf rank takes on a useful value ; even in the Max Cut case, if cost is rank 1 with 1's on diagonal, the theorem's conclusion is vacuous.
	%Do check the reference they give: ``On the rank of a cograph'', but seems to yield nothing useful.
	%}}
	If the right hand side evaluates to $\ell$, then there exists $\nu$ such that $M = C-\calA^*(\nu)$ and $\nulll(M) = \ell$. Writing $C = M + \calA^*(\nu)$, we find that
	\begin{align}
	C & \in \calN_\ell + \im(\calA^*)
	\end{align}
	where the $+$ is a set-sum and $\calN_\ell$ denotes the set of symmetric matrices of size $n$ with nullity $\ell$. This set has dimension %\TODO{Explain?}
	\begin{align}
	\dim \calN_\ell = \frac{n(n+1)}{2} - \frac{\ell(\ell+1)}{2},
	\end{align}
	whereas $\dim \im(\calA^*) = \rank(\calA^*) \leq m$. Assume the right hand side of~\eqref{eq:baseinequality} evaluates to $p$ or more. Then, a fortiori,
	\begin{align}
	C \in \bigcup_{\ell = p,\ldots,n} \calN_\ell + \im(\calA^*).
	\label{eq:CbelongsUnion}
	\end{align}
	The set on the right hand side contains all ``bad'' $C$'s, that is, those for which~\eqref{eq:baseinequality} offers no information about the rank of $Y$. The dimension of that set is bounded as follows, using that the dimension of a finite union is at most the maximal dimension, and the dimension of a finite sum of sets is at most the sum of the set dimensions:
	%	\footnote{``If a subset of $\Rn$ has Hausdorff dimension less than $n$ then it is a null set with respect to $n$-dimensional Lebesgue measure.'' \url{https://en.wikipedia.org/wiki/Lebesgue\_measure\#Null\_sets} --- The Hausdorff dimension is bounded above by the upper box dimension, \url{https://en.wikipedia.org/wiki/Minkowski\%E2\%80\%93Bouligand\_dimension\#Relations\_to\_the\_Hausdorff\_dimension} --- and the upper box dimension has the properties we require \url{https://en.wikipedia.org/wiki/Minkowski\%E2\%80\%93Bouligand\_dimension\#Properties}. Then I assumed that the upper box dimension of a manifold is equal to its Hausdorff dimension... See also \url{https://mathoverflow.net/questions/44192/fourier-dimension-of-the-sum-of-sets}, where it says (without reference) that $dim_H(A+B) <= dim_H(A) + dim_B(B)$ ; Then, it's probably true that both dimensions on the rhs are equal to their usual dimensions, and that settles it. --- see \url{https://en.wikipedia.org/wiki/Dimension\_of\_an\_algebraic\_variety}}
	\begin{align*}
	\dim\left( \bigcup_{\ell = p,\ldots,n} \calN_\ell + \im(\calA^*) \right) \leq \dim\left( \calN_p + \im(\calA^*) \right) \leq \frac{n(n+1)}{2} - \frac{p(p+1)}{2} + m.
	\end{align*}
	Since $C\in\Snn$ lives in a space of dimension $\frac{n(n+1)}{2}$, almost no $C$ verifies~\eqref{eq:CbelongsUnion} if
	\begin{align*}
	\frac{n(n+1)}{2} - \frac{p(p+1)}{2} + m < \frac{n(n+1)}{2}.
	\end{align*}
	Hence, if $\frac{p(p+1)}{2} > m$, then, for almost all $C$, critical points verify $\rank(Y) < p$.
\end{proof}
Theorem~\ref{thm:masterthm} follows as a corollary of Corollary~\ref{cor:rankdefsocpopt} and Lemma~\ref{lem:criticalptsrankdeficient}.

\section{Numerical experiments}\label{sec:XP}

As an example, we run five different solvers on~\eqref{eq:maxcutsdp} with a collection of graphs used in~\citep{sdplr,burer2005local} known as the Gset.\footnote{Downloaded from: \url{http://web.stanford.edu/~yyye/yyye/Gset/} on June 6, 2016.} The solvers are as follows, all run in Matlab. The first three are based on a low-rank factorization while the last two are interior point methods (IPM).
\begin{itemize}
	\item[] \texttt{Manopt} runs the Riemannian Trust-Region method on~\eqref{eq:maxcutBM}, via the Manopt toolbox~\citep{manopt}, with $p = \left\lceil \frac{\sqrt{8n+1}}{2} \right\rceil$ and random initialization. The number of inner iterations allowed to solve the trust-region subproblem is 500. The solver returns when $\frac12 \|\grad f(Y)\| = \|SY\|\leq 10^{-6}$. Code is in Matlab.
	\item[] \texttt{Manopt+} runs the same algorithm as above, but with $p$ increasing from $2$ to $\left\lceil \frac{\sqrt{8n+1}}{2} \right\rceil$ in 5 steps. The point $Y$ computed at a lower $p$ is appended with columns of i.i.d.\ random Gaussian variables with standard deviation $10^{-5}$ and mean 0, then rows are normalized to produce $Y_+$: the initial point for the next value of $p$. The randomization allows to escape near-saddle points (in practice). Code is in Matlab.
	\item[] \texttt{SDPLR} runs the original Burer--Monteiro algorithm implemented by its authors~\citep{sdplr}. Code is in C interfaced through C-mex.
	\item[] \texttt{HRVW} runs an IPM whose implementation is tailored to~\eqref{eq:maxcutsdp}, implemented by its authors~\citep{helmberg1996ipmsdp}. Code is in Matlab.
	\item[] \texttt{CVX} runs SDPT3~\citep{sdpt3} on~\eqref{eq:maxcutsdp} via CVX~\citep{cvx}. Code is in C interfaced through C-mex.
\end{itemize}
After the solvers return, we project their answers to the feasible set. Manopt and SDPLR return a matrix $Y$: it is sufficient to normalize each row to ensure $X = YY\transpose$ is feasible for~\eqref{eq:maxcutsdp} (for Manopt, this step is not necessary). HRVW and CVX return a symmetric matrix $X$. We compute its Cholesky factorization $X = RR\transpose$---if $X$ is not positive semidefinite, we first project using an eigenvalue decomposition. Then, each row of $R$ is normalized so that $X = RR\transpose$ is feasible for~\eqref{eq:maxcutsdp}. Computation time required for these projections is not included in the timings.

We report three metrics for each graph and each solver.
\begin{itemize}
	\item[] Cut bound: a bound on the maximal cut value (lower is better). If $C$ is the adjacency matrix of the graph and $D$ is the degree matrix, then $L = D-C$ is the Laplacian and $\max_{X} \frac{1}{4}\inner{L}{X}$ s.t.\ $\diag(X) = \1, X \succeq 0$ is a bound on the maximal cut. Using Lemma~\ref{lem:optimgap} applied to~\eqref{eq:maxcutsdp}, a candidate optimizer $X$ yields a bound $\frac{1}{4}\inner{L}{X} - \frac{n}{4}\lambdamin(S)$.
	\item[] $\lambdamin(S)$: by Lemma~\ref{lem:optimgap}, this is a measure of optimality for $X$ (feasible), where $S = C - \diag(\diag(CX))$. It is nonpositive and must be as close to 0 as possible. We compute it using bisection and the Cholesky factorization to ensure accuracy.
	\item[] Time: computation time in seconds for the solver to run\footnote{Matlab R2015a on $2 \times 6$ cores processors with hyperthreading, Intel(R) Xeon(R) CPU E5-2640 @ 2.50GHz, 256Gb RAM, Springdale Linux 6.} (this excludes time taken to project the solution to the feasible set and to compute the reported metrics.)
\end{itemize}

Based on the results reported in Table~\ref{tab:maxcutXPrudy}, we make the following main observations: (i) the Manopt approach (optimization on manifolds, also advocated in~\citep{journee2010low}) consistently reaches high accuracy solutions, being often orders of magnitude more accurate than other methods, as judged from $\lambdamin(S)$; (ii) incremental rank solvers (Manopt+ and SDPLR) are often the fastest solvers for large instances; and (iii) the tailored IPM HRVW is faster and typically more accurate than the IPM called by CVX (which is generic software). The latter point hints that one must be careful in dismissing IPMs based on experiments using generic software, although it remains clear from Table~\ref{tab:maxcutXPrudy} that IPMs scale poorly compared to the low-rank factorization methods tested here. In particular, CVX runs into memory trouble for the larger problem instances reported.\footnote{On Graph 77, running CVX leads to Matlab error ``Number of elements exceeds maximum flint $2^{53}-1$.''} To save time, we did not run CVX on the largest graphs.

\section{Numerical experiments: results}

%%\begin{table}[h]
%	\centering
%     \begin{adjustbox}{width=\textwidth,center}
% \begin{adjustbox}{center}
\begin{longtable}{lllllll}
	
	% Caption
	\caption{Results of the experiments described in Section~\ref{sec:XP}.}
	\label{tab:maxcutXPrudy} \\
	% Header
	\ttfamily Graph & \ttfamily Metric & \ttfamily Manopt & \ttfamily Manopt+ & \ttfamily SDPLR & \ttfamily HRVW & \ttfamily CVX \\ \hline
	\endfirsthead
	
	% Header
	\ttfamily Graph & \ttfamily Metric & \ttfamily Manopt & \ttfamily Manopt+ & \ttfamily SDPLR & \ttfamily HRVW & \ttfamily CVX \\ \hline
	
	\endhead
	
	% Data
	Graph 1 & Cut bound & 12083.2 & 12083.2 & 12083.2 & 12083.2 & 12083.2 \\\nopagebreak
\quad 800 nodes & $\lambdamin(S)$ & $-3 \cdot 10^{-11}$ & $-2 \cdot 10^{-11}$ & $-9 \cdot 10^{-6}$ & $-2 \cdot 10^{-5}$ & $-3 \cdot 10^{-6}$ \\\nopagebreak
\quad 19176 edges & Time [s] & 2.1 & 3.2 & 6.6 & 1.9 & 35.0 \\
\hline

Graph 2 & Cut bound & 12089.4 & 12089.4 & 12089.4 & 12089.4 & 12089.4 \\\nopagebreak
\quad 800 nodes & $\lambdamin(S)$ & $-2 \cdot 10^{-10}$ & $-8 \cdot 10^{-12}$ & $-5 \cdot 10^{-6}$ & $-3 \cdot 10^{-5}$ & $-7 \cdot 10^{-7}$ \\\nopagebreak
\quad 19176 edges & Time [s] & 1.6 & 3.1 & 7.8 & 2.0 & 33.7 \\
\hline

Graph 3 & Cut bound & 12084.3 & 12084.3 & 12085.5 & 12084.3 & 12084.3 \\\nopagebreak
\quad 800 nodes & $\lambdamin(S)$ & $-3 \cdot 10^{-11}$ & $-1 \cdot 10^{-11}$ & $-6 \cdot 10^{-3}$ & $-4 \cdot 10^{-5}$ & $-2 \cdot 10^{-6}$ \\\nopagebreak
\quad 19176 edges & Time [s] & 2.1 & 4.5 & 9.8 & 2.0 & 34.0 \\
\hline

Graph 4 & Cut bound & 12111.5 & 12111.5 & 12111.5 & 12111.5 & 12111.5 \\\nopagebreak
\quad 800 nodes & $\lambdamin(S)$ & $-2 \cdot 10^{-11}$ & $-2 \cdot 10^{-10}$ & $-1 \cdot 10^{-5}$ & $-3 \cdot 10^{-5}$ & $-6 \cdot 10^{-6}$ \\\nopagebreak
\quad 19176 edges & Time [s] & 1.8 & 3.2 & 10.6 & 2.2 & 33.7 \\
\hline

Graph 5 & Cut bound & 12099.9 & 12099.9 & 12099.9 & 12099.9 & 12099.9 \\\nopagebreak
\quad 800 nodes & $\lambdamin(S)$ & $-3 \cdot 10^{-12}$ & $-8 \cdot 10^{-12}$ & $-1 \cdot 10^{-5}$ & $-3 \cdot 10^{-5}$ & $-1 \cdot 10^{-6}$ \\\nopagebreak
\quad 19176 edges & Time [s] & 1.5 & 2.5 & 6.7 & 2.2 & 33.7 \\
\hline

Graph 6 & Cut bound & 2656.2 & 2656.2 & 2660.8 & 2656.2 & 2656.2 \\\nopagebreak
\quad 800 nodes & $\lambdamin(S)$ & $-4 \cdot 10^{-12}$ & $-8 \cdot 10^{-12}$ & $-2 \cdot 10^{-2}$ & $-7 \cdot 10^{-6}$ & $-9 \cdot 10^{-6}$ \\\nopagebreak
\quad 19176 edges & Time [s] & 1.4 & 2.6 & 5.5 & 2.4 & 34.1 \\
\hline

Graph 7 & Cut bound & 2489.3 & 2489.3 & 2489.3 & 2489.3 & 2489.3 \\\nopagebreak
\quad 800 nodes & $\lambdamin(S)$ & $-2 \cdot 10^{-11}$ & $-2 \cdot 10^{-11}$ & $-1 \cdot 10^{-5}$ & $-9 \cdot 10^{-6}$ & $-4 \cdot 10^{-7}$ \\\nopagebreak
\quad 19176 edges & Time [s] & 6.4 & 2.6 & 5.9 & 2.0 & 35.7 \\
\hline

Graph 8 & Cut bound & 2506.9 & 2506.9 & 2506.9 & 2506.9 & 2506.9 \\\nopagebreak
\quad 800 nodes & $\lambdamin(S)$ & $-5 \cdot 10^{-12}$ & $-9 \cdot 10^{-12}$ & $-4 \cdot 10^{-5}$ & $-1 \cdot 10^{-5}$ & $-1 \cdot 10^{-6}$ \\\nopagebreak
\quad 19176 edges & Time [s] & 1.2 & 1.8 & 10.6 & 2.2 & 34.0 \\
\hline

Graph 9 & Cut bound & 2528.7 & 2528.7 & 2528.7 & 2528.7 & 2528.7 \\\nopagebreak
\quad 800 nodes & $\lambdamin(S)$ & $-1 \cdot 10^{-9}$ & $-8 \cdot 10^{-12}$ & $-8 \cdot 10^{-6}$ & $-1 \cdot 10^{-5}$ & $-1 \cdot 10^{-6}$ \\\nopagebreak
\quad 19176 edges & Time [s] & 0.9 & 1.8 & 5.7 & 2.4 & 34.8 \\
\hline

Graph 10 & Cut bound & 2485.1 & 2485.1 & 2485.1 & 2485.1 & 2485.1 \\\nopagebreak
\quad 800 nodes & $\lambdamin(S)$ & $-5 \cdot 10^{-11}$ & $-8 \cdot 10^{-12}$ & $-6 \cdot 10^{-6}$ & $-8 \cdot 10^{-6}$ & $-2 \cdot 10^{-6}$ \\\nopagebreak
\quad 19176 edges & Time [s] & 1.2 & 1.6 & 5.3 & 2.1 & 33.9 \\
\hline

Graph 11 & Cut bound & 629.2 & 629.2 & 629.2 & 629.2 & 629.2 \\\nopagebreak
\quad 800 nodes & $\lambdamin(S)$ & $-3 \cdot 10^{-9}$ & $-7 \cdot 10^{-12}$ & $-5 \cdot 10^{-6}$ & $-1 \cdot 10^{-6}$ & $-4 \cdot 10^{-8}$ \\\nopagebreak
\quad 1600 edges & Time [s] & 13.6 & 13.6 & 3.9 & 2.0 & 31.5 \\
\hline

Graph 12 & Cut bound & 623.9 & 623.9 & 623.9 & 623.9 & 623.9 \\\nopagebreak
\quad 800 nodes & $\lambdamin(S)$ & $-1 \cdot 10^{-10}$ & $-4 \cdot 10^{-12}$ & $-3 \cdot 10^{-6}$ & $-3 \cdot 10^{-6}$ & $-9 \cdot 10^{-8}$ \\\nopagebreak
\quad 1600 edges & Time [s] & 8.8 & 7.3 & 1.9 & 2.0 & 31.7 \\
\hline

Graph 13 & Cut bound & 647.1 & 647.1 & 647.1 & 647.1 & 647.1 \\\nopagebreak
\quad 800 nodes & $\lambdamin(S)$ & $-1 \cdot 10^{-9}$ & $-2 \cdot 10^{-12}$ & $-2 \cdot 10^{-6}$ & $-2 \cdot 10^{-6}$ & $-1 \cdot 10^{-7}$ \\\nopagebreak
\quad 1600 edges & Time [s] & 6.9 & 6.7 & 1.3 & 2.2 & 31.4 \\
\hline

Graph 14 & Cut bound & 3191.6 & 3191.6 & 3191.6 & 3191.6 & 3191.6 \\\nopagebreak
\quad 800 nodes & $\lambdamin(S)$ & $-1 \cdot 10^{-10}$ & $-3 \cdot 10^{-12}$ & $-3 \cdot 10^{-5}$ & $-3 \cdot 10^{-5}$ & $-1 \cdot 10^{-6}$ \\\nopagebreak
\quad 4694 edges & Time [s] & 1.5 & 5.3 & 4.4 & 2.5 & 34.1 \\
\hline

Graph 15 & Cut bound & 3171.6 & 3171.6 & 3171.6 & 3171.6 & 3171.6 \\\nopagebreak
\quad 800 nodes & $\lambdamin(S)$ & $-1 \cdot 10^{-10}$ & $-5 \cdot 10^{-12}$ & $-6 \cdot 10^{-6}$ & $-5 \cdot 10^{-6}$ & $-3 \cdot 10^{-7}$ \\\nopagebreak
\quad 4661 edges & Time [s] & 3.4 & 6.5 & 5.4 & 3.2 & 34.6 \\
\hline

Graph 16 & Cut bound & 3175.0 & 3175.0 & 3175.1 & 3175.0 & 3175.0 \\\nopagebreak
\quad 800 nodes & $\lambdamin(S)$ & $-9 \cdot 10^{-12}$ & $-2 \cdot 10^{-12}$ & $-6 \cdot 10^{-4}$ & $-1 \cdot 10^{-5}$ & $-6 \cdot 10^{-7}$ \\\nopagebreak
\quad 4672 edges & Time [s] & 6.6 & 6.2 & 3.8 & 3.1 & 34.8 \\
\hline

Graph 17 & Cut bound & 3171.3 & 3171.3 & 3171.5 & 3171.3 & 3171.3 \\\nopagebreak
\quad 800 nodes & $\lambdamin(S)$ & $-5 \cdot 10^{-12}$ & $-2 \cdot 10^{-12}$ & $-1 \cdot 10^{-3}$ & $-1 \cdot 10^{-5}$ & $-1 \cdot 10^{-7}$ \\\nopagebreak
\quad 4667 edges & Time [s] & 6.1 & 6.3 & 3.5 & 2.9 & 34.5 \\
\hline

Graph 18 & Cut bound & 1166.0 & 1166.0 & 1166.0 & 1166.0 & 1166.0 \\\nopagebreak
\quad 800 nodes & $\lambdamin(S)$ & $-4 \cdot 10^{-12}$ & $-3 \cdot 10^{-12}$ & $-3 \cdot 10^{-6}$ & $-4 \cdot 10^{-6}$ & $-1 \cdot 10^{-6}$ \\\nopagebreak
\quad 4694 edges & Time [s] & 1.8 & 2.9 & 4.2 & 3.2 & 35.1 \\
\hline

Graph 19 & Cut bound & 1082.0 & 1082.0 & 1082.0 & 1082.0 & 1082.0 \\\nopagebreak
\quad 800 nodes & $\lambdamin(S)$ & $-4 \cdot 10^{-10}$ & $-4 \cdot 10^{-12}$ & $-4 \cdot 10^{-6}$ & $-3 \cdot 10^{-6}$ & $-8 \cdot 10^{-7}$ \\\nopagebreak
\quad 4661 edges & Time [s] & 1.9 & 2.8 & 4.3 & 3.4 & 34.5 \\
\hline

Graph 20 & Cut bound & 1111.4 & 1111.4 & 1112.1 & 1111.4 & 1111.4 \\\nopagebreak
\quad 800 nodes & $\lambdamin(S)$ & $-2 \cdot 10^{-12}$ & $-3 \cdot 10^{-12}$ & $-3 \cdot 10^{-3}$ & $-4 \cdot 10^{-6}$ & $-2 \cdot 10^{-6}$ \\\nopagebreak
\quad 4672 edges & Time [s] & 2.8 & 3.7 & 2.9 & 3.6 & 34.1 \\
\hline

Graph 21 & Cut bound & 1104.3 & 1104.3 & 1104.3 & 1104.3 & 1104.3 \\\nopagebreak
\quad 800 nodes & $\lambdamin(S)$ & $-2 \cdot 10^{-11}$ & $-6 \cdot 10^{-12}$ & $-4 \cdot 10^{-6}$ & $-2 \cdot 10^{-6}$ & $-6 \cdot 10^{-6}$ \\\nopagebreak
\quad 4667 edges & Time [s] & 2.7 & 4.3 & 3.5 & 3.7 & 34.1 \\
\hline

Graph 22 & Cut bound & 14135.9 & 14135.9 & 14136.0 & 14135.9 & 14137.2 \\\nopagebreak
\quad 2000 nodes & $\lambdamin(S)$ & $-8 \cdot 10^{-12}$ & $-8 \cdot 10^{-12}$ & $-3 \cdot 10^{-5}$ & $-3 \cdot 10^{-5}$ & $-2 \cdot 10^{-3}$ \\\nopagebreak
\quad 19990 edges & Time [s] & 5.5 & 4.9 & 22.5 & 25.7 & 177.7 \\
\hline

Graph 23 & Cut bound & 14142.1 & 14142.1 & 14142.1 & 14142.1 & 14143.5 \\\nopagebreak
\quad 2000 nodes & $\lambdamin(S)$ & $-2 \cdot 10^{-11}$ & $-3 \cdot 10^{-11}$ & $-8 \cdot 10^{-6}$ & $-3 \cdot 10^{-5}$ & $-3 \cdot 10^{-3}$ \\\nopagebreak
\quad 19990 edges & Time [s] & 7.0 & 9.1 & 16.3 & 23.8 & 182.8 \\
\hline

Graph 24 & Cut bound & 14140.9 & 14140.9 & 14140.9 & 14140.9 & 14142.1 \\\nopagebreak
\quad 2000 nodes & $\lambdamin(S)$ & $-1 \cdot 10^{-11}$ & $-7 \cdot 10^{-12}$ & $-1 \cdot 10^{-5}$ & $-2 \cdot 10^{-5}$ & $-2 \cdot 10^{-3}$ \\\nopagebreak
\quad 19990 edges & Time [s] & 4.5 & 5.7 & 24.3 & 24.8 & 173.3 \\
\hline

Graph 25 & Cut bound & 14144.2 & 14144.2 & 14148.8 & 14144.2 & 14145.8 \\\nopagebreak
\quad 2000 nodes & $\lambdamin(S)$ & $-1 \cdot 10^{-9}$ & $-9 \cdot 10^{-12}$ & $-9 \cdot 10^{-3}$ & $-9 \cdot 10^{-6}$ & $-3 \cdot 10^{-3}$ \\\nopagebreak
\quad 19990 edges & Time [s] & 4.8 & 18.1 & 16.7 & 23.8 & 175.0 \\
\hline

Graph 26 & Cut bound & 14132.9 & 14132.9 & 14132.9 & 14132.9 & 14134.2 \\\nopagebreak
\quad 2000 nodes & $\lambdamin(S)$ & $-7 \cdot 10^{-12}$ & $-1 \cdot 10^{-11}$ & $-4 \cdot 10^{-6}$ & $-2 \cdot 10^{-5}$ & $-3 \cdot 10^{-3}$ \\\nopagebreak
\quad 19990 edges & Time [s] & 6.8 & 6.5 & 14.4 & 23.1 & 177.6 \\
\hline

Graph 27 & Cut bound & 4141.7 & 4141.7 & 4145.0 & 4141.7 & 4143.1 \\\nopagebreak
\quad 2000 nodes & $\lambdamin(S)$ & $-1 \cdot 10^{-11}$ & $-7 \cdot 10^{-12}$ & $-7 \cdot 10^{-3}$ & $-9 \cdot 10^{-6}$ & $-3 \cdot 10^{-3}$ \\\nopagebreak
\quad 19990 edges & Time [s] & 3.7 & 4.4 & 10.8 & 23.5 & 175.9 \\
\hline

Graph 28 & Cut bound & 4100.8 & 4100.8 & 4100.8 & 4100.8 & 4102.2 \\\nopagebreak
\quad 2000 nodes & $\lambdamin(S)$ & $-2 \cdot 10^{-9}$ & $-6 \cdot 10^{-12}$ & $-3 \cdot 10^{-5}$ & $-7 \cdot 10^{-6}$ & $-3 \cdot 10^{-3}$ \\\nopagebreak
\quad 19990 edges & Time [s] & 3.0 & 8.0 & 19.6 & 26.5 & 176.8 \\
\hline

Graph 29 & Cut bound & 4208.9 & 4208.9 & 4208.9 & 4208.9 & 4210.0 \\\nopagebreak
\quad 2000 nodes & $\lambdamin(S)$ & $-2 \cdot 10^{-11}$ & $-2 \cdot 10^{-11}$ & $-5 \cdot 10^{-6}$ & $-2 \cdot 10^{-6}$ & $-2 \cdot 10^{-3}$ \\\nopagebreak
\quad 19990 edges & Time [s] & 12.2 & 8.3 & 17.7 & 24.5 & 180.6 \\
\hline

Graph 30 & Cut bound & 4215.4 & 4215.4 & 4215.4 & 4215.4 & 4216.6 \\\nopagebreak
\quad 2000 nodes & $\lambdamin(S)$ & $-7 \cdot 10^{-11}$ & $-6 \cdot 10^{-12}$ & $-5 \cdot 10^{-6}$ & $-6 \cdot 10^{-6}$ & $-2 \cdot 10^{-3}$ \\\nopagebreak
\quad 19990 edges & Time [s] & 19.8 & 10.5 & 11.6 & 25.2 & 176.7 \\
\hline

Graph 31 & Cut bound & 4116.7 & 4116.7 & 4119.1 & 4116.7 & 4118.0 \\\nopagebreak
\quad 2000 nodes & $\lambdamin(S)$ & $-2 \cdot 10^{-11}$ & $-5 \cdot 10^{-12}$ & $-5 \cdot 10^{-3}$ & $-7 \cdot 10^{-6}$ & $-3 \cdot 10^{-3}$ \\\nopagebreak
\quad 19990 edges & Time [s] & 4.1 & 8.9 & 16.2 & 26.2 & 170.6 \\
\hline

Graph 32 & Cut bound & 1567.6 & 1567.6 & 1567.6 & 1567.6 & 1567.8 \\\nopagebreak
\quad 2000 nodes & $\lambdamin(S)$ & $-2 \cdot 10^{-10}$ & $-8 \cdot 10^{-12}$ & $-1 \cdot 10^{-6}$ & $-1 \cdot 10^{-6}$ & $-3 \cdot 10^{-4}$ \\\nopagebreak
\quad 4000 edges & Time [s] & 45.6 & 25.4 & 13.9 & 21.7 & 142.6 \\
\hline

Graph 33 & Cut bound & 1544.3 & 1544.3 & 1544.3 & 1544.3 & 1544.4 \\\nopagebreak
\quad 2000 nodes & $\lambdamin(S)$ & $-7 \cdot 10^{-10}$ & $-5 \cdot 10^{-12}$ & $-1 \cdot 10^{-6}$ & $-9 \cdot 10^{-7}$ & $-1 \cdot 10^{-4}$ \\\nopagebreak
\quad 4000 edges & Time [s] & 31.2 & 17.3 & 9.9 & 23.0 & 141.2 \\
\hline

Graph 34 & Cut bound & 1546.7 & 1546.7 & 1546.7 & 1546.7 & 1546.8 \\\nopagebreak
\quad 2000 nodes & $\lambdamin(S)$ & $-1 \cdot 10^{-9}$ & $-5 \cdot 10^{-12}$ & $-2 \cdot 10^{-6}$ & $-1 \cdot 10^{-6}$ & $-2 \cdot 10^{-4}$ \\\nopagebreak
\quad 4000 edges & Time [s] & 31.3 & 22.0 & 7.7 & 23.6 & 143.9 \\
\hline

Graph 35 & Cut bound & 8014.7 & 8014.7 & 8014.7 & 8014.7 & 8015.3 \\\nopagebreak
\quad 2000 nodes & $\lambdamin(S)$ & $-1 \cdot 10^{-9}$ & $-4 \cdot 10^{-11}$ & $-5 \cdot 10^{-6}$ & $-9 \cdot 10^{-6}$ & $-1 \cdot 10^{-3}$ \\\nopagebreak
\quad 11778 edges & Time [s] & 19.4 & 17.4 & 26.0 & 34.5 & 187.7 \\
\hline

Graph 36 & Cut bound & 8006.0 & 8006.0 & 8006.0 & 8006.0 & 8006.6 \\\nopagebreak
\quad 2000 nodes & $\lambdamin(S)$ & $-9 \cdot 10^{-10}$ & $-3 \cdot 10^{-11}$ & $-1 \cdot 10^{-5}$ & $-2 \cdot 10^{-5}$ & $-1 \cdot 10^{-3}$ \\\nopagebreak
\quad 11766 edges & Time [s] & 12.0 & 36.9 & 41.1 & 37.0 & 193.3 \\
\hline

Graph 37 & Cut bound & 8018.6 & 8018.6 & 8019.4 & 8018.6 & 8019.5 \\\nopagebreak
\quad 2000 nodes & $\lambdamin(S)$ & $-2 \cdot 10^{-10}$ & $-1 \cdot 10^{-11}$ & $-1 \cdot 10^{-3}$ & $-1 \cdot 10^{-5}$ & $-2 \cdot 10^{-3}$ \\\nopagebreak
\quad 11785 edges & Time [s] & 11.2 & 15.4 & 38.4 & 35.2 & 191.1 \\
\hline

Graph 38 & Cut bound & 8015.0 & 8015.0 & 8015.0 & 8015.0 & 8015.5 \\\nopagebreak
\quad 2000 nodes & $\lambdamin(S)$ & $-1 \cdot 10^{-10}$ & $-1 \cdot 10^{-11}$ & $-2 \cdot 10^{-5}$ & $-1 \cdot 10^{-5}$ & $-1 \cdot 10^{-3}$ \\\nopagebreak
\quad 11779 edges & Time [s] & 13.1 & 14.2 & 44.7 & 37.5 & 193.0 \\
\hline

Graph 39 & Cut bound & 2877.6 & 2877.6 & 2877.8 & 2877.6 & 2878.4 \\\nopagebreak
\quad 2000 nodes & $\lambdamin(S)$ & $-4 \cdot 10^{-9}$ & $-7 \cdot 10^{-12}$ & $-3 \cdot 10^{-4}$ & $-4 \cdot 10^{-6}$ & $-2 \cdot 10^{-3}$ \\\nopagebreak
\quad 11778 edges & Time [s] & 16.9 & 12.2 & 31.9 & 39.3 & 195.8 \\
\hline

Graph 40 & Cut bound & 2864.8 & 2864.8 & 2866.2 & 2864.8 & 2865.6 \\\nopagebreak
\quad 2000 nodes & $\lambdamin(S)$ & $-1 \cdot 10^{-11}$ & $-2 \cdot 10^{-11}$ & $-3 \cdot 10^{-3}$ & $-3 \cdot 10^{-6}$ & $-2 \cdot 10^{-3}$ \\\nopagebreak
\quad 11766 edges & Time [s] & 9.2 & 9.4 & 40.8 & 40.9 & 189.0 \\
\hline

Graph 41 & Cut bound & 2865.2 & 2865.2 & 2868.1 & 2865.2 & 2865.8 \\\nopagebreak
\quad 2000 nodes & $\lambdamin(S)$ & $-4 \cdot 10^{-10}$ & $-1 \cdot 10^{-11}$ & $-6 \cdot 10^{-3}$ & $-4 \cdot 10^{-6}$ & $-1 \cdot 10^{-3}$ \\\nopagebreak
\quad 11785 edges & Time [s] & 5.3 & 8.6 & 30.8 & 40.9 & 189.8 \\
\hline

Graph 42 & Cut bound & 2946.3 & 2946.3 & 2948.3 & 2946.3 & 2947.0 \\\nopagebreak
\quad 2000 nodes & $\lambdamin(S)$ & $-9 \cdot 10^{-12}$ & $-7 \cdot 10^{-12}$ & $-4 \cdot 10^{-3}$ & $-6 \cdot 10^{-6}$ & $-1 \cdot 10^{-3}$ \\\nopagebreak
\quad 11779 edges & Time [s] & 7.9 & 8.1 & 32.9 & 41.8 & 188.4 \\
\hline

Graph 43 & Cut bound & 7032.2 & 7032.2 & 7032.2 & 7032.2 & 7033.2 \\\nopagebreak
\quad 1000 nodes & $\lambdamin(S)$ & $-3 \cdot 10^{-12}$ & $-4 \cdot 10^{-12}$ & $-6 \cdot 10^{-6}$ & $-2 \cdot 10^{-5}$ & $-4 \cdot 10^{-3}$ \\\nopagebreak
\quad 9990 edges & Time [s] & 1.9 & 2.3 & 3.6 & 3.8 & 36.4 \\
\hline

Graph 44 & Cut bound & 7027.9 & 7027.9 & 7029.2 & 7027.9 & 7029.4 \\\nopagebreak
\quad 1000 nodes & $\lambdamin(S)$ & $-1 \cdot 10^{-8}$ & $-3 \cdot 10^{-12}$ & $-5 \cdot 10^{-3}$ & $-2 \cdot 10^{-5}$ & $-6 \cdot 10^{-3}$ \\\nopagebreak
\quad 9990 edges & Time [s] & 2.9 & 3.9 & 3.7 & 3.6 & 38.0 \\
\hline

Graph 45 & Cut bound & 7024.8 & 7024.8 & 7024.8 & 7024.8 & 7025.9 \\\nopagebreak
\quad 1000 nodes & $\lambdamin(S)$ & $-1 \cdot 10^{-9}$ & $-5 \cdot 10^{-12}$ & $-2 \cdot 10^{-5}$ & $-8 \cdot 10^{-6}$ & $-5 \cdot 10^{-3}$ \\\nopagebreak
\quad 9990 edges & Time [s] & 1.3 & 6.1 & 4.9 & 3.5 & 37.4 \\
\hline

Graph 46 & Cut bound & 7029.9 & 7029.9 & 7029.9 & 7029.9 & 7030.8 \\\nopagebreak
\quad 1000 nodes & $\lambdamin(S)$ & $-2 \cdot 10^{-10}$ & $-3 \cdot 10^{-12}$ & $-2 \cdot 10^{-5}$ & $-1 \cdot 10^{-5}$ & $-4 \cdot 10^{-3}$ \\\nopagebreak
\quad 9990 edges & Time [s] & 12.9 & 2.3 & 3.1 & 3.7 & 38.3 \\
\hline

Graph 47 & Cut bound & 7036.7 & 7036.7 & 7036.7 & 7036.7 & 7037.8 \\\nopagebreak
\quad 1000 nodes & $\lambdamin(S)$ & $-8 \cdot 10^{-10}$ & $-9 \cdot 10^{-12}$ & $-1 \cdot 10^{-5}$ & $-1 \cdot 10^{-5}$ & $-5 \cdot 10^{-3}$ \\\nopagebreak
\quad 9990 edges & Time [s] & 10.4 & 4.1 & 8.2 & 3.8 & 39.2 \\
\hline

Graph 48 & Cut bound & 6000.0 & 6000.0 & 6000.0 & 6000.0 & 6000.0 \\\nopagebreak
\quad 3000 nodes & $\lambdamin(S)$ & $4 \cdot 10^{-16}$ & $3 \cdot 10^{-16}$ & $-6 \cdot 10^{-10}$ & $-3 \cdot 10^{-6}$ & $5 \cdot 10^{-18}$ \\\nopagebreak
\quad 6000 edges & Time [s] & 2.8 & 4.3 & 3.5 & 47.7 & 307.3 \\
\hline

Graph 49 & Cut bound & 6000.0 & 6000.0 & 6000.0 & 6000.0 & 6000.0 \\\nopagebreak
\quad 3000 nodes & $\lambdamin(S)$ & $4 \cdot 10^{-16}$ & $4 \cdot 10^{-16}$ & $-1 \cdot 10^{-9}$ & $-3 \cdot 10^{-6}$ & $-4 \cdot 10^{-16}$ \\\nopagebreak
\quad 6000 edges & Time [s] & 3.9 & 5.1 & 4.9 & 46.1 & 299.7 \\
\hline

Graph 50 & Cut bound & 5988.2 & 5988.2 & 5988.2 & 5988.2 & 5988.2 \\\nopagebreak
\quad 3000 nodes & $\lambdamin(S)$ & $-2 \cdot 10^{-12}$ & $-1 \cdot 10^{-14}$ & $-1 \cdot 10^{-7}$ & $-3 \cdot 10^{-6}$ & $2 \cdot 10^{-16}$ \\\nopagebreak
\quad 6000 edges & Time [s] & 6.0 & 5.0 & 5.4 & 45.7 & 318.4 \\
\hline

Graph 51 & Cut bound & 4006.3 & 4006.3 & 4006.3 & 4006.3 & 4006.9 \\\nopagebreak
\quad 1000 nodes & $\lambdamin(S)$ & $-2 \cdot 10^{-9}$ & $-4 \cdot 10^{-12}$ & $-8 \cdot 10^{-6}$ & $-1 \cdot 10^{-5}$ & $-3 \cdot 10^{-3}$ \\\nopagebreak
\quad 5909 edges & Time [s] & 5.8 & 7.8 & 10.7 & 5.4 & 41.4 \\
\hline

Graph 52 & Cut bound & 4009.6 & 4009.6 & 4010.0 & 4009.6 & 4010.2 \\\nopagebreak
\quad 1000 nodes & $\lambdamin(S)$ & $-4 \cdot 10^{-12}$ & $-9 \cdot 10^{-12}$ & $-1 \cdot 10^{-3}$ & $-5 \cdot 10^{-6}$ & $-2 \cdot 10^{-3}$ \\\nopagebreak
\quad 5916 edges & Time [s] & 6.4 & 8.8 & 6.5 & 5.2 & 39.6 \\
\hline

Graph 53 & Cut bound & 4009.7 & 4009.7 & 4009.7 & 4009.7 & 4010.5 \\\nopagebreak
\quad 1000 nodes & $\lambdamin(S)$ & $-1 \cdot 10^{-10}$ & $-1 \cdot 10^{-11}$ & $-6 \cdot 10^{-6}$ & $-1 \cdot 10^{-5}$ & $-3 \cdot 10^{-3}$ \\\nopagebreak
\quad 5914 edges & Time [s] & 4.2 & 8.5 & 8.3 & 5.0 & 39.1 \\
\hline

Graph 54 & Cut bound & 4006.2 & 4006.2 & 4006.2 & 4006.2 & 4006.9 \\\nopagebreak
\quad 1000 nodes & $\lambdamin(S)$ & $-2 \cdot 10^{-10}$ & $-3 \cdot 10^{-12}$ & $-3 \cdot 10^{-5}$ & $-5 \cdot 10^{-6}$ & $-3 \cdot 10^{-3}$ \\\nopagebreak
\quad 5916 edges & Time [s] & 2.9 & 6.6 & 6.1 & 4.8 & 39.1 \\
\hline

Graph 55 & Cut bound & 11039.5 & 11039.5 & 11039.5 & 11039.5 & 11039.7 \\\nopagebreak
\quad 5000 nodes & $\lambdamin(S)$ & $-2 \cdot 10^{-12}$ & $-3 \cdot 10^{-12}$ & $-5 \cdot 10^{-6}$ & $-6 \cdot 10^{-6}$ & $-2 \cdot 10^{-4}$ \\\nopagebreak
\quad 12498 edges & Time [s] & 26.6 & 20.6 & 22.2 & 411.4 & 1588.0 \\
\hline

Graph 56 & Cut bound & 4760.0 & 4760.0 & 4760.0 & 4760.0 & 4760.3 \\\nopagebreak
\quad 5000 nodes & $\lambdamin(S)$ & $-7 \cdot 10^{-12}$ & $-2 \cdot 10^{-12}$ & $-1 \cdot 10^{-5}$ & $-2 \cdot 10^{-6}$ & $-3 \cdot 10^{-4}$ \\\nopagebreak
\quad 12498 edges & Time [s] & 20.1 & 16.3 & 32.9 & 475.9 & 1550.1 \\
\hline

Graph 57 & Cut bound & 3885.5 & 3885.5 & 3885.5 & 3885.5 & 3885.7 \\\nopagebreak
\quad 5000 nodes & $\lambdamin(S)$ & $-1 \cdot 10^{-9}$ & $-8 \cdot 10^{-12}$ & $-2 \cdot 10^{-6}$ & $-2 \cdot 10^{-6}$ & $-1 \cdot 10^{-4}$ \\\nopagebreak
\quad 10000 edges & Time [s] & 218.0 & 78.8 & 38.3 & 269.8 & 1012.4 \\
\hline

Graph 58 & Cut bound & 20136.2 & 20136.2 & 20138.1 & 20136.2 & 20136.7 \\\nopagebreak
\quad 5000 nodes & $\lambdamin(S)$ & $-3 \cdot 10^{-9}$ & $-5 \cdot 10^{-11}$ & $-2 \cdot 10^{-3}$ & $-7 \cdot 10^{-6}$ & $-4 \cdot 10^{-4}$ \\\nopagebreak
\quad 29570 edges & Time [s] & 55.4 & 44.0 & 321.5 & 497.9 & 1865.7 \\
\hline

Graph 59 & Cut bound & 7312.3 & 7312.3 & 7315.0 & 7312.3 & 7313.0 \\\nopagebreak
\quad 5000 nodes & $\lambdamin(S)$ & $-7 \cdot 10^{-12}$ & $-3 \cdot 10^{-11}$ & $-2 \cdot 10^{-3}$ & $-4 \cdot 10^{-6}$ & $-5 \cdot 10^{-4}$ \\\nopagebreak
\quad 29570 edges & Time [s] & 51.3 & 35.6 & 353.1 & 511.3 & 1869.0 \\
\hline

Graph 60 & Cut bound & 15222.3 & 15222.3 & 15222.3 & 15222.3 & 15222.6 \\\nopagebreak
\quad 7000 nodes & $\lambdamin(S)$ & $-3 \cdot 10^{-11}$ & $-4 \cdot 10^{-12}$ & $-2 \cdot 10^{-5}$ & $-2 \cdot 10^{-6}$ & $-2 \cdot 10^{-4}$ \\\nopagebreak
\quad 17148 edges & Time [s] & 58.6 & 30.9 & 63.6 & 1326.9 & 3581.9 \\
\hline

Graph 61 & Cut bound & 6828.1 & 6828.1 & 6828.2 & 6828.1 & 6828.4 \\\nopagebreak
\quad 7000 nodes & $\lambdamin(S)$ & $-2 \cdot 10^{-11}$ & $-4 \cdot 10^{-12}$ & $-7 \cdot 10^{-5}$ & $-2 \cdot 10^{-6}$ & $-2 \cdot 10^{-4}$ \\\nopagebreak
\quad 17148 edges & Time [s] & 113.4 & 40.2 & 55.8 & 1263.3 & 3795.6 \\
\hline

Graph 62 & Cut bound & 5430.9 & 5430.9 & 5430.9 & 5430.9 & 5431.1 \\\nopagebreak
\quad 7000 nodes & $\lambdamin(S)$ & $-1 \cdot 10^{-9}$ & $-6 \cdot 10^{-11}$ & $-9 \cdot 10^{-7}$ & $-2 \cdot 10^{-6}$ & $-1 \cdot 10^{-4}$ \\\nopagebreak
\quad 14000 edges & Time [s] & 813.8 & 242.8 & 110.8 & 862.4 & 2124.3 \\
\hline

Graph 63 & Cut bound & 28244.4 & 28244.4 & 28245.9 & 28244.4 & 28245.0 \\\nopagebreak
\quad 7000 nodes & $\lambdamin(S)$ & $-7 \cdot 10^{-9}$ & $-8 \cdot 10^{-9}$ & $-8 \cdot 10^{-4}$ & $-9 \cdot 10^{-6}$ & $-3 \cdot 10^{-4}$ \\\nopagebreak
\quad 41459 edges & Time [s] & 238.9 & 97.6 & 663.0 & 1454.7 & 4583.9 \\
\hline

Graph 64 & Cut bound & 10465.9 & 10465.9 & 10466.6 & 10465.9 & 10466.6 \\\nopagebreak
\quad 7000 nodes & $\lambdamin(S)$ & $-3 \cdot 10^{-9}$ & $-2 \cdot 10^{-11}$ & $-4 \cdot 10^{-4}$ & $-5 \cdot 10^{-6}$ & $-4 \cdot 10^{-4}$ \\\nopagebreak
\quad 41459 edges & Time [s] & 140.4 & 109.5 & 1014.8 & 1609.4 & 4439.8 \\
\hline

Graph 65 & Cut bound & 6205.5 & 6205.5 & 6205.5 & 6205.5 & 6205.7 \\\nopagebreak
\quad 8000 nodes & $\lambdamin(S)$ & $-1 \cdot 10^{-9}$ & $-1 \cdot 10^{-11}$ & $-6 \cdot 10^{-7}$ & $-1 \cdot 10^{-6}$ & $-1 \cdot 10^{-4}$ \\\nopagebreak
\quad 16000 edges & Time [s] & 567.2 & 168.5 & 154.4 & 1075.2 & 2861.5 \\
\hline

Graph 66 & Cut bound & 7077.2 & 7077.2 & 7077.2 & 7077.2 & 7077.4 \\\nopagebreak
\quad 9000 nodes & $\lambdamin(S)$ & $-2 \cdot 10^{-9}$ & $-5 \cdot 10^{-11}$ & $-2 \cdot 10^{-7}$ & $-9 \cdot 10^{-7}$ & $-6 \cdot 10^{-5}$ \\\nopagebreak
\quad 18000 edges & Time [s] & 762.6 & 215.3 & 218.1 & 1525.7 & 3915.7 \\
\hline

Graph 67 & Cut bound & 7744.4 & 7744.4 & 7744.4 & 7744.4 & - \\\nopagebreak
\quad 10000 nodes & $\lambdamin(S)$ & $-1 \cdot 10^{-9}$ & $-3 \cdot 10^{-11}$ & $-3 \cdot 10^{-7}$ & $-1 \cdot 10^{-6}$ & - \\\nopagebreak
\quad 20000 edges & Time [s] & 816.4 & 339.0 & 267.3 & 2005.4 & - \\
\hline

Graph 70 & Cut bound & 9861.5 & 9861.5 & 9861.5 & 9861.5 & - \\\nopagebreak
\quad 10000 nodes & $\lambdamin(S)$ & $-2 \cdot 10^{-10}$ & $-6 \cdot 10^{-13}$ & $-2 \cdot 10^{-6}$ & $-2 \cdot 10^{-6}$ & - \\\nopagebreak
\quad 9999 edges & Time [s] & 143.3 & 82.9 & 102.2 & 3167.3 & - \\
\hline

Graph 72 & Cut bound & 7808.5 & 7808.5 & 7808.5 & 7808.5 & - \\\nopagebreak
\quad 10000 nodes & $\lambdamin(S)$ & $-6 \cdot 10^{-10}$ & $-8 \cdot 10^{-12}$ & $-8 \cdot 10^{-7}$ & $-1 \cdot 10^{-6}$ & - \\\nopagebreak
\quad 20000 edges & Time [s] & 720.8 & 262.6 & 199.0 & 1902.7 & - \\
\hline

Graph 77 & Cut bound & 11045.7 & 11045.7 & 11045.7 & 11045.7 & - \\\nopagebreak
\quad 14000 nodes & $\lambdamin(S)$ & $-8 \cdot 10^{-10}$ & $-4 \cdot 10^{-11}$ & $-7 \cdot 10^{-7}$ & $-1 \cdot 10^{-6}$ & - \\\nopagebreak
\quad 28000 edges & Time [s] & 1578.5 & 513.0 & 515.1 & 5249.1 & - \\
\hline

Graph 81 & Cut bound & 15656.2 & 15656.2 & 15656.2 & 15656.2 & - \\\nopagebreak
\quad 20000 nodes & $\lambdamin(S)$ & $-5 \cdot 10^{-10}$ & $-6 \cdot 10^{-11}$ & $-1 \cdot 10^{-6}$ & $-3 \cdot 10^{-6}$ & - \\\nopagebreak
\quad 40000 edges & Time [s] & 4152.8 & 1539.7 & 1035.6 & 16576.6 & - \\
\hline

	%        % Data
	%        \multirow{3}{*}{Graph id, 1000 nodes, 19900 edges} & \multicolumn{1}{l}{Cost $\inner{C}{X}$} & \multicolumn{1}{l}{XXX} & \multicolumn{1}{l}{XXX} & \multicolumn{1}{l}{XXX} & \multicolumn{1}{l}{XXX} \\\cline{2-6}
	%        & \multicolumn{1}{l}{$\lambdamin(S)$} & \multicolumn{1}{l}{XXX} & \multicolumn{1}{l}{XXX} & \multicolumn{1}{l}{XXX} & \multicolumn{1}{l}{XXX} \\\cline{2-6}
	%        & \multicolumn{1}{l}{Time [s]} & \multicolumn{1}{l}{XXX} & \multicolumn{1}{l}{XXX} & \multicolumn{1}{l}{XXX} & \multicolumn{1}{l}{XXX} \\\cline{2-6}
	%        \hline
	%        
\end{longtable}

\section{Regularity assumption} \label{apdx:regularity}

Originally, Theorems~\ref{thm:masterthm} and~\ref{thm:approxtolerance} had the assumption that the search space of the factorized problem,
\begin{align*}
\calM & = \{ Y \in \Rnp : \calA(YY\transpose) = b \},
\end{align*}
is a manifold. From this assumption, we stated incorrectly that the tangent space at $Y$ of $\calM$, denoted by $\T_Y\calM$, is given by~\eqref{eq:TYM}:
\begin{align*}
\T_Y\calM & = \{ \dot Y \in \Rnp : \calA(\dot Y Y\transpose + Y \dot Y\transpose) = 0 \}.
\end{align*}
This identity is used in a number of places of the proofs. In general, $\calM$ being an embedded submanifold of $\Rnp$ only implies the left hand side is included in the right hand side.\footnote{If $\dot Y \in \T_Y\calM$, by definition, there exists a smooth curve $\gamma \colon \reals \to \calM$ such that $\gamma(0) = Y$ and $\gamma'(0) = \dot Y$. Since $\gamma(t) \in \calM$ for all $t$, we have $\calA(\gamma(t)\gamma(t)\transpose) = b$ for all $t$. Differentiating on both sides with respect to $t$ and evaluating at $0$ gives $\calA(\dot YY\transpose + Y\dot Y\transpose) = 0$.} Below, we give an example where $\calM$ is a manifold yet the two sets are not equal.

In order to restore equality, we strengthened the assumption, requiring constraint qualifications to hold at all feasible points (see~\eqref{eq:CQ}):
\begin{align*}
\forall Y \in \calM, \quad A_1Y, \ldots, A_mY \textrm{ are linearly independent in } \Rnp,
\end{align*}
where $A_i$, $i = 1, \ldots, m$, are the symmetric constraint matrices such that $\calA(X)_i = \inner{A_i}{X}$. This ensures the map $\Phi(Y) = \calA(YY\transpose) - b$ is full rank on $\calM = \Phi^{-1}(0)$, from which it follows by a standard result in differential geometry (see for example~\citep[Cor.~5.14]{lee2012smoothmanifolds}) that $\calM$ is a smooth embedded submanifold of $\Rnp$ of dimension $np - m$. Then, the left hand side of~\eqref{eq:TYM} has dimension $np - m$, and it is included in the right hand side, which itself is a linear space of dimension $np - m$, so that they are equal.

%Strengthening the assumption that $\calM$ is a manifold by assuming instead that~\eqref{eq:CQ} holds (which implies $\calM$ is a manifold) restores equality in~\eqref{eq:TYM}, allowing the proofs to go through. All examples of interest, and in particular all examples treated in the paper, satisfy this stronger assumption. Furthermore, the main message of the paper, namely the relevance of smoothness in low-rank factorizations of SDPs, remains valid.

We now describe an SDP such that $\calM$ is indeed a manifold, yet~\eqref{eq:TYM} does not hold. Consider $n = 2, m = 2$, $b = (1, 1)\transpose$ and
\begin{align*}
A_1 & = \begin{pmatrix}
1 & 0 \\ 0 & 1
\end{pmatrix}, & A_2 & = \begin{pmatrix}
1 & 0 \\ 0 & \frac{1}{4}
\end{pmatrix}.
\end{align*}
The search space of the SDP,
\begin{align*}
\calC = \{X = X\transpose \in \reals^{n\times n} : \calA(X) = b, X \succeq 0\} = \left\{ \begin{pmatrix}
1 & 0 \\ 0 & 0
\end{pmatrix} \right\},
\end{align*}
is degenerate but it is compact. Furthermore, the set $\calM$ is a smooth manifold for $p = 1$:
\begin{align*}
\calM_{p=1} & = \left\{ Y = \begin{pmatrix}
y_1 \\ y_2
\end{pmatrix} \in \reals^2 : y_1^2 + y_2^2 = 1 \textrm{ and } y_1^2 + \frac{1}{4}y_2^2 = 1 \right\} = \{ (1, 0)\transpose, (-1, 0)\transpose \}.
\end{align*}
The dimension of the manifold is 0, so that $\T_Y\calM = \{0\}$ for all $Y\in\calM$. Consider now the right hand side of~\eqref{eq:TYM},
\begin{align*}
K_Y & = \{ \dot Y \in \Rnp : \calA(\dot Y Y\transpose + Y \dot Y\transpose) = 0 \} = \{ \dot Y \in \Rnp : \innersmall{A_1Y}{\dot Y} = \innersmall{A_2Y}{\dot Y} = 0 \}.
\end{align*}
These are the vectors orthogonal to $A_1Y, A_2Y$. For $Y = (\pm 1, 0)\transpose$, we get $A_1Y = A_2Y = (\pm 1, 0)\transpose$: they are colinear, so $K_Y$ has dimension 1 at all $Y \in \calM$: $\T_Y\calM \neq K_Y$.

Similarly, at $p = 2$, the set $\calM$ becomes a circle embedded in $\reals^4$:
\begin{align*}
\calM_{p=2} & = \left\{ Y = \begin{pmatrix}
y_1 & y_2 \\ y_3 & y_4
\end{pmatrix} \in \reals^{2\times 2} : y_1^2 + y_2^2 + y_3^2 + y_4^2 = 1 \textrm{ and } y_1^2 + y_2^2 + \frac{1}{4}(y_3^2 + y_4^2) = 1 \right\} \\
& = \left\{ Y = \begin{pmatrix}
y_1 & y_2 \\ y_3 & y_4
\end{pmatrix} \in \reals^{2\times 2} : y_1^2 + y_2^2 = 1 \textrm{ and } y_3 = y_4 = 0 \right\}.
\end{align*}
This manifold has dimension 1 (and so do all its tangent spaces). Yet, $K_Y$ has dimension 3 for all $Y \in \calM$. Indeed, we can parameterize $\calM_{p=2}$ as the matrices
\begin{align*}
\begin{pmatrix}
\cos\theta & \sin\theta \\ 0 & 0
\end{pmatrix}
\end{align*}
for all $\theta \in \reals$. It is easy to verify that $A_1Y = A_2Y \neq 0$ for all $Y\in\calM_{p=2}$, so that the codimension of $K_Y$ is 1, here too in disagreement with $\T_Y\calM$. Notice also that in this example we have $\frac{p(p+1)}{2} > m$.

%As a side note, we remark that, in this example, for any cost matrix $C$, the cost function $f(Y) = \inner{C}{YY\transpose}$ is constant on $\calM$. In that sense, this is not a counter-example to Theorem 2 as stated: it is a counter-example to one step in the proof.

\end{document}